\documentclass[11pt, draft]{amsart} 
\usepackage{amsmath, amsthm, amssymb, mathtools}

\allowdisplaybreaks[1]

\title[Relations among some conjectures]
{Relations among some conjectures on the M\"obius function and the Riemann zeta-function}

\author[S. Inoue]{Sh\={o}ta Inoue}

\address{Graduate School of Mathematics, Nagoya University,
Furocho, Chikusaku, Nagoya 464-8602, Japan}
\email{m16006w@math.nagoya-u.ac.jp}

\keywords{M\"{o}bius Function, Riesz Mean of an Arithmetic Function, Riemann Zeta-Function, Riemann Hypothesis, Simple Zero Conjecture, 
Mertens Hypothesis, Weak Mertens Hypothesis, Gonek-Hejhal Conjecture, Linear Independence Conjecture}

\subjclass[2010]{Primary 11M26; Secondary 11M06}

\theoremstyle{break} 
\newtheorem{theorem}{Theorem}
\newtheorem{lemma}{Lemma}
\newtheorem{corollary}{Corollary}
\newtheorem{conjecture}{Conjecture}

\newtheorem{proposition}{Proposition}

\theoremstyle{plain}
\newtheorem{fact}{Fact}

\theoremstyle{remark}
\newtheorem{remark}{Remark}

\theoremstyle{definition}
\newtheorem*{acknow*}{Acknowledgments}

\numberwithin{equation}{section}

\newcommand{\e}{\varepsilon}
\newcommand{\s}{\sigma}

\newcommand{\normal}{\normalfont}

\renewcommand{\a}{\alpha}
\renewcommand{\l}{\left}
\renewcommand{\r}{\right}
\renewcommand{\d}{\displaystyle}
\renewcommand{\Re}{\mathrm{Re}}

\begin{document}

\begin{abstract}
We discuss the multiplicity of the non-trivial zeros of the Riemann zeta-function and the summatory function $M(x)$ of the M\"obius function.
The purpose of this paper is to consider two open problems under some conjectures.
One is that whether all zeros of the Riemann zeta-function are simple or not. 
The other problem is that whether $M(x) \ll x^{1/2}$ holds or not. 
First, we consider the former problem.
It is known that the assertion $M(x) = o(x^{1/2}\log{x})$ is a sufficient condition for the proof of the simplicity of zeros.
However, proving this assertion is presently difficult.
Therefore, we consider another sufficient condition for the simplicity of zeros that is weaker than the above assertion
in terms of the Riesz mean $M_{\tau}(x) = {\Gamma(1+\tau)}^{-1}\sum_{n \leq x}\mu(n)(1 - \frac{n}{x})^{\tau}$. 
We conclude that the assertion $M_{\tau}(x) = o(x^{1/2}\log{x})$ for a non-negative fixed $\tau$ is a
sufficient condition for the simplicity of zeros. 
Also, we obtain an explicit formula for $M_{\tau}(x)$. 
By observing the formula, we propose a conjecture, in which $\tau$ is not fixed, but depends on $x$.
This conjecture also gives a sufficient condition, which seems easier to approach, for the simplicity of zeros.
Next, we consider the latter problem. 
Many mathematicians believe that the estimate $ M(x) \ll x^{1/2}$ fails, but this is not yet disproved.
In this paper we study the mean values $\int_{1}^{x}\frac{M(u)}{u^{\kappa}}du$ for any real $\kappa$ under the weak Mertens Hypothesis
$\int_{1}^{x}( M(u)/u)^2du \ll \log{x}$. 
We obtain the upper bound of $\int_{1}^{x}\frac{M(u)}{u^{\kappa}}du$ under the weak Mertens Hypothesis.
We also have $\Omega$-result of this integral unconditionally, and so we find that
the upper bound which is obtained in this paper of this integral is the best possible estimation.
\end{abstract}

\maketitle


                                  \section{\bf Introduction and statement of results}


We define the summatory function $M(x) := \sum_{n \leq x}\mu(n)$, where $\mu(n)$ is the M\"obius function.
The M\"obius function is defined by 
\begin{align*}
\mu(n) 
= \l\{
\begin{array}{ll}
	1 		& \text{if \; $n = 1$},\\
	(-1)^{k} 	& \text{if \; $n$ is the product of $k$ different primes},\\
	0 		& \text{otherwise}.
\end{array}
\r.
\end{align*}
We discuss the Riesz means
\begin{align*}
M_\tau(x) := \frac{1}{\Gamma(\tau+1)}\sum_{n \leq x}\mu(n)\l( 1 - \frac{n}{x} \r)^\tau
\end{align*}
for $\tau \geq 0$ in this paper. We consider both cases when $\tau$ is constant and when $\tau$ depends on $x$.


It is well known that the original Mertens conjecture states that 

\begin{align}	\label{OMH}
	|M(x)| \leq x^{1/2}
\end{align}
for $x \geq 1$. 
However, this conjecture (\ref{OMH}) was disproved by A. M. Odlyzko and H. J. J. te Riele in \cite{OR}.
Recently, D. G. Best and T. S. Trudgian \cite{BT} disproved that $|M(x)| < cx^{1/2}$ for $c < 1.6383$.
Many mathematicians believe that even
\begin{align}		\label{MH}
M(x) \ll x^{1/2}
\end{align} 
fails, but this is as yet unproved.
The following fact is a crucial reason why many mathematicians believe that the inequality (\ref{MH}) is false.

\begin{fact}	\label{flm}
If the Linear Independence Conjecture is true, then the inequality (\ref{MH}) is false.
\end{fact}

The Fact \ref{flm} was shown by A. E. Ingham \cite{In} in 1942.
We note that the Linear Independence Conjecture is the following conjecture.


\begin{conjecture}[The Linear Independence Conjecture]

Assume that the Riemann zeta-function $\zeta(s)$ satisfies the Riemann Hypothesis.
Then the positive ordinates of distinct zeros are linearly independent over $\mathbb{Q}$.
\end{conjecture}


The summatory function $M(x)$ is important in the study of prime numbers.
Actually, the Riemann Hypothesis is equivalent to the inequality $M(x) \ll x^{1/2+\varepsilon}$ for any positive $\varepsilon$, 
which is weaker than the inequality (\ref{MH}).

\footnote{The author cannot follow Bartz's proof of (\ref{exfoM_0}). However, instead of Lemma 1 in the paper \cite{B},
if we use Lemma \ref{zlb} in the present paper, it is possible to follow her proof.}
K. M. Bartz \cite{B} showed the following explicit formula for $M(x)$:
\begin{align}	\label{exfoM_0}
M(x) = 
&\lim_{T_\nu \rightarrow \infty}\sum_{|\gamma| < T_{\nu}}\frac{1}{(m(\rho)-1)!}\lim_{s \rightarrow \rho}\frac{d^{m(\rho)-1}}{ds^{m(\rho)-1}}
			\left((s-\rho)^{m(\rho)}\frac{x^{s}}{s\zeta(s)}\right)\\ \nonumber
&-2+\sum_{n=1}^{\infty}\frac{(-1)^{n-1}(x/2\pi)^{-2n}}{(2n)!n\zeta(2n+1)},
\end{align}
unconditionally, where $\rho = \beta + i\gamma$ denotes a non-trivial zero of $\zeta(s)$,
and $m(\rho)$ is the multiplicity of $\rho$.
We see that the multiplicity of zeros of the Riemann zeta-function is important to find the upper bound of $M(x)$.
In fact, it is known that
\begin{align*}
M(x) = \Omega_\pm\left(x^{1/2}(\log{x})^{m-1}\right)
\end{align*}
if $\zeta(s)$ has a zero of multiplicity $m$ (see p. 467 in H. L. Montgomery and R. C. Vaughan \cite{MV}). 
First we discuss whether the claim that all zeros of the Riemann zeta-function are simple (we abbreviate this as (SZC)) is true or not.
Today, many mathematicians believe that (SZC) is true and there are many works under (SZC).
H. M. Bui and D. R. Heath-Brown \cite{BB} showed that the rate of simple zeros of Riemann zeta-function 
is larger than about 70.37\% under the Riemann Hypothesis, and D. A. Goldston and S. M. Gonek \cite{GG} showed 
the upper bound 
$
m(\rho) \leq \left(\frac{1}{2} + o(1)\right)\frac{\log{|\gamma|}}{\log{\log{|\gamma|}}} 
$
under the Riemann Hypothesis.

Our first result is the following theorem for the multiplicity of the zeros of $\zeta(s)$.


\begin{theorem} \label{SZCApuRH}
Assume the Riemann Hypothesis. 
If there exists a positive number $\tau$ satisfying
$M_{\tau}(x) = o\l(x^{1/2}(\log{x})^{\alpha}\r)$, then the inequality
\begin{align*}
m(\rho) < \alpha + 1
\end{align*}
holds for any non-trivial zero $\rho$ of the Riemann zeta-function.
\end{theorem}


This theorem is shown by considering the $\Omega$-result on $M_{\tau}(x)$. 
The parameter $\tau$ is fixed in this theorem. 
However, the author believes that a similar result would hold for $\tau$ depending on $x$ under the assumption of
a certain upper bound of $M_{\tau}(x)$.
In order to discuss such situation (see Proposition \ref{SZCuMycon} below), we first prove the following explicit formula.

\begin{remark}
In what follows, $\varepsilon$ and $\delta$ denote any arbitrarily small positive numbers, 
not necessarily the same ones at each occurrence.
\end{remark}


\begin{theorem}   \label{exfoM}

There exists a sequence $\{T_\nu\}$ tending to infinity and satisfying
\begin{align*}
M_\tau(x) =& 
\lim_{T_{\nu} \rightarrow \infty}\sum_{|\gamma| < T_{\nu}}\frac{1}{(m(\rho)-1)!}\lim_{s \rightarrow \rho}\frac{d^{m(\rho)-1}}{ds^{m(\rho)-1}}
			\left((s-\rho)^{m(\rho)}\frac{x^{s}}{\zeta(s)}\frac{\Gamma(s)}{\Gamma(1+\tau+s)}\right)\\
&+\sum_{l=0}^{\infty}\underset{s=-l}{\rm Res}\l(\frac{x^s}{\zeta(s)}\frac{\Gamma(s)}{\Gamma(1+\tau+s)}\r)
\end{align*}
for any numbers $\tau > 0$, $x > 0$, 
and the series in the first term is uniformly convergent with respect to $x$ on any compact subset $K \subset (0, \infty)$ for $\tau \geq \delta$, 
and the series in the second term is absolutely and uniformly convergent with respect to $x \geq \delta$ for $\tau \geq 0$.
Furthermore, we have
\begin{align*}
\sum_{l=0}^{\infty}\underset{s=-l}{\rm Res}\l(\frac{x^s}{\zeta(s)}\frac{\Gamma(s)}{\Gamma(1+\tau+s)}\r) \ll 1.
\end{align*}
\end{theorem}


From this formula we can prove various interesting consequences. 
Assume the Riemann Hypothesis and let $\rho_1 = \frac{1}{2} + i\gamma_1$ be a multiple zero of the Riemann zeta-function. 
By Leibniz's rule, we find that
\begin{align*}
M_\tau(x) = &
2x^{1/2}(\log{x})^{m\left(\rho_1\right)-1}m(\rho_1)\Re\l(\frac{\Gamma(\rho_1)}{\zeta^{(m(\rho_1))}(\rho_1)\Gamma(1+\tau+\rho_1)}x^{i\gamma_1}\r)\\
&+ \frac{2x^{1/2}}{(m(\rho_1)-1)!}
\sum_{l=0}^{m(\rho_1)-2}
\begin{pmatrix}
	m\l(\rho_1 \r)-1 \\
	l
\end{pmatrix}
(\log{x})^{l} \times\\
&\Re\l(\lim_{s \rightarrow \rho_1}\frac{d^{m(\rho_1)-1-l}}{ds^{m(\rho_1)-1-l}}\left( (s - \rho_1)^{m(\rho_1)}
\frac{\Gamma(s)}{\zeta(s)\Gamma(1+\tau+s)} \right) x^{i\gamma_1} \r) \nonumber\\ 
&+\lim_{T_\nu \rightarrow \infty}\underset{|\gamma| \not= |\gamma_1|}{\sum_{|\gamma| < T_{\nu}}}
\frac{1}{(m(\rho)-1)!}\lim_{s \rightarrow \rho}\frac{d^{m(\rho)-1}}{ds^{m(\rho)-1}}
			\left((s-\rho)^{m(\rho)}\frac{x^{s}}{\zeta(s)}\frac{\Gamma(s)}{\Gamma(1+\tau+s)}\right)\\
&+\sum_{l=0}^{\infty}\underset{s=-l}{\rm Res}\l(\frac{x^s}{\zeta(s)}\frac{\Gamma(s)}{\Gamma(1+\tau+s)}\r)\\
&=: 2x^{1/2}(\log{x})^{m\left(\rho_1\right)-1}m(\rho_1)
\Re\l(\frac{x^{i\gamma_1}\Gamma(\rho_1)}{\zeta^{(m(\rho_1))}(\rho_1)\Gamma(1+\tau+\rho_1)}\r) 
+ Y_{\tau, \rho_1}(x),
\end{align*}
say.  We believe that the first term of the right-hand side of the above formula dominates the behavior of $M_{\tau}(x)$. 
We propose the following conjecture.


\begin{conjecture} \label{Mycon}
Let $\rho$ be any non-trivial zeros of the Riemann zeta-function.
For any monotone positive valued function $\tau = \tau(x)$, we have
\begin{align*}
M_{\tau}(x) = \Omega\left(x^{1/2}(\log{x})^{m(\rho)-1}(\tau/e)^{-\tau-1}\right).
\end{align*}
\end{conjecture}


We can obtain the following result under Conjecture \ref{Mycon}.


\begin{proposition} \label{SZCuMycon}
Assume Conjecture \ref{Mycon}. If there exists a monotone positive valued function 
$\tau = \tau(x) \leq \frac{\log{\log{x}}}{\log{\log{\log{x}}}}(\alpha + o(1))$ satisfying
$M_{\tau}(x) \ll x^{1/2}(\log{x})^{\beta}$, then the inequality
\begin{align*}
m(\rho) \leq \alpha + \beta + 1
\end{align*}
holds for any non-trivial zero $\rho$ of the Riemann zeta-function.\\
In particular, if $\alpha + \beta < 1$, then (SZC) holds.
\end{proposition}


Now, we can obtain the following result on the bound of $M_{\tau}(x)$ under the Riemann Hypothesis.


\begin{theorem} \label{ASZC}
Assume the Riemann Hypothesis. Then there exists a positive constant $C_0 > 0$ such that
\begin{align*}
M_{\tau}(x) \ll x^{1/2}
\end{align*}
holds for any $\tau = \tau(x) \geq \frac{C_0\log{\log{x}}}{\log{\log{\log{\log{x}}}}}$.
\end{theorem}


In view of this theorem, it is important to study $M_{\tau}(x)$ in the case when $\tau$ depends on $x$.

Note that, however, the study of $M_{\tau}(x)$ for fixed $\tau$ is also not worthless. 
By Theorem \ref{SZCApuRH}, we see that whether $M_{\tau}(x) = o(x^{1/2}\log{x})$ holds or not is crucial to prove (SZC), 
and in this paper, actually a stronger assertion $M_{\tau}(x) \ll x^{1/2}$ holds for any $\tau > 1/2$
under certain two situations; one is the situation 
where Gonek-Hejhal Conjecture and the Riemann Hypothesis hold, and the other is the situation where the weak Mertens Hypothesis holds.



Next, we study the Riesz mean $M_{\tau}(x)$ more closely by the explicit formula in Theorem \ref{exfoM}.
If (SZC) is true, then we have
\begin{align*}
M_{\tau}(x) = \lim_{T_{\nu} \rightarrow \infty}\sum_{|\gamma| < T_{\nu}}\frac{x^{\rho}}{\zeta'(\rho)}\frac{\Gamma(\rho)}{\Gamma(1+\tau+\rho)}
+\sum_{l=0}^{\infty}\underset{s=-l}{\mathrm{Res}}\left( \frac{x^{s}}{\zeta(s)}\frac{\Gamma(s)}{\Gamma(1 + s + \tau)} \right)
\end{align*}
from Theorem \ref{exfoM}.
If (SZC) is solved, then a natural next problem is to find some lower bound of $\zeta'(\rho)$.
This problem is difficult because it is deeply connected with the distribution of zeros of the Riemann zeta-function.
These are the reasons why we now mention the following conjecture:


\begin{conjecture}[The Gonek-Hejhal Hypothesis]	\label{GHH} 
\begin{align*}
J_{\lambda}(T)& := \sum_{0 < \gamma \leq T}\left|\zeta'(\rho)\right|^{2\lambda} 
                  \asymp T(\log{T})^{(\lambda+1)^2}          
\end{align*}
for any $\lambda > -\frac{3}{2}$ under (SZC).
\end{conjecture}


From the viewpoints different from each other, Gonek \cite{G} and Hejhal \cite{H} independently suggested this conjecture.
In fact, $J_{0}(T) = N(T) \asymp T\log{T}$ (see Lemma \ref{gzdbz} in this paper), and Gonek \cite{G} showed 
\begin{align}	\label{Gonek-estimate}
J_{-1}(T) \gg T.
\end{align}
In addition, he suggested the asymptotic formula
\begin{align}	\label{sharpGHH}
J_{-1}(T) \sim \frac{3}{\pi^3}T.
\end{align}
Moreover, from the view point of the theory of random matrices, C. P. Hughes, J. P. Keating and N. O'Connell \cite {HKC} suggested that,
for $\lambda > -3/2$, we have
\begin{align}
J_{\lambda}(T) \sim \frac{G^2(\lambda+2)}{G(2\lambda + 3)}a_\lambda\frac{T}{2\pi}\l( \log\frac{T}{2\pi} \r)^{(\lambda+1)^2},
\end{align}
where
\begin{align*}
a_\lambda = \prod_{p}\l( 1 - \frac{1}{p} \r)^{\lambda^2}\l( \sum_{m=0}^{\infty}\l( \frac{\Gamma(m+\lambda)}{m!\Gamma(\lambda)} \r)^2p^{-m} \r),
\end{align*}
and $G$ is Barnes' function defined by
\begin{align*}
G(z+1) = (2\pi)^{z/2}\exp\l( -\frac{1}{2}(z^2+\gamma z^2 + z) \r)\prod_{n=1}^{\infty}\l( \l( 1 + \frac{z}{n} \r)^n e^{-z+z^2/2n} \r).
\end{align*}
Here $\gamma$ denotes Euler's constant.
Hence we may say that Conjecture \ref{GHH} is supported from various points of view.

We can obtain the following corollary by assuming the Riemann Hypothesis and Conjecture \ref{GHH} at $\lambda = -1$.


\begin{corollary} \label{dzb}

We assume the Riemann Hypothesis. If $J_{-1}(T) \ll T$ holds, then we have 
\begin{align*}
M_\tau(x) \ll \l\{
\begin{array}{ll}
	x^{1/2}					& \text{\normal{if} \; $1 \ll \tau(x)$}, \vspace{1mm} \\
	x^{1/2}/\tau^{3/2}			& \text{\normal{if} \; $(\log{x})^{-1} \ll \tau(x) = o(1)$}, \vspace{1mm} \\
	x^{1/2}(\log{x})^{3/2} 			& \text{\normal{if} \; $0 \leq \tau(x) =  o\l((\log{x})^{-1}\r)$}
\end{array}
\r.
\end{align*}
for any positive number $x > 2$.
\end{corollary}


The case $\tau \equiv 0$ of Corollary \ref{dzb} is also the same as Theorem 1 (i) in N. Ng's paper \cite{ng1}.
We obtain the result analogous to the inequality (\ref{MH}) for $M_{\tau}(x)$ 
when $\tau \gg 1$ by Corollary \ref{dzb}. 
Furthermore, if the above assumption $J_{-1}(T) \ll T$ is replaced by the assumption $J_{-1/2}(T) \ll T(\log{T})^{1/4}$, 
then we have
\begin{align*}
M_{\tau}(x) \ll \l\{
\begin{array}{ll}
	x^{1/2}					& \text{if \; $1 \ll \tau(x)$}, \vspace{1mm} \\
	x^{1/2}/\tau^{5/4}			& \text{if \; $(\log{x})^{-1} \ll \tau(x) = o(1)$}, \vspace{1mm} \\
	x^{1/2}(\log{x})^{5/4} 			& \text{if \; $0 \leq \tau(x) =  o\l((\log{x})^{-1}\r)$}.
\end{array}
\r.
\end{align*}
The case $\tau \equiv 0$ of the above result is mentioned by Ng \cite{ng1}.

Also, we consider the following conjecture:


\begin{conjecture}[The weak Mertens Hypothesis]	\label{WMH}
\begin{align}
\int_{1}^{x}\left(\frac{M(u)}{u}\right)^2du \ll \log{x}.
\end{align}
\end{conjecture}


E. C. Titchmarsh carefully discussed various fascinating facts under this conjecture in \cite[Section 14.\ 28]{T}.
For example, Conjecture \ref{WMH} implies that the Riemann Hypothesis and (SZC) are true, 
and that $\d{\sum_{\rho}\frac{1}{|\rho \zeta'(\rho)|^2}}$ is convergent. 

Also, Ng \cite{ng1} showed that the Riemann Hypothesis and $J_{-1}(T) \ll T$ imply the weak Mertens Hypothesis.
Moreover, he proved that the Riemann Hypothesis and $J_{-1}(T) \ll T$ imply
\begin{align}	\label{SWMH}
\int_{1}^{x}\left(\frac{M(u)}{u}\right)^2du \sim \log{x}\sum_{\gamma}\frac{1}{|\rho\zeta'(\rho)|^2}.
\end{align}

Now, we obtain the following results under Conjecture \ref{WMH}.


\begin{corollary}	\label{AWMH}

We assume the weak Mertens Hypothesis. Let $\tau > 1/2$ be a real number. Then, we have
\begin{align*}
M_{\tau}(x) = x^{1/2}\sum_{\gamma}\frac{x^{i\gamma}}{\zeta'(\rho)}\frac{\Gamma(\rho)}{\Gamma(\rho+\tau+1)}
+\sum_{l=0}^{\infty}\underset{s=-l}{\rm Res}\l(\frac{x^s}{\zeta(s)}\frac{\Gamma(s)}{\Gamma(1+\tau+s)}\r),
\end{align*}
and the series in the first term is uniformly and absolutely convergent with respect to $x \in (0, \infty)$.
In particular, we have 
\begin{align*}
	M_{\tau}(x) \ll_{\varepsilon} x^{1/2}
\end{align*}
for any $\tau \geq 1/2+\varepsilon$.
\end{corollary}


This corollary implies the inequality, which is analogous to the inequality (\ref{MH}) for $M_{\tau}(x)$ 
with $\tau > 1/2$ under the weak Mertens Hypothesis.

The Riesz mean has the recurrence formula
\begin{align*}
\int_{1}^{x}u^{\tau-1}M_{\tau-1}(u)du = x^{\tau}M_{\tau}(x)
\end{align*}
for $\tau \geq 1$.
Thus we can study the integral of $M(x)$ by $M_{\tau}(x)$.
The following results are under this principle. 
Let $A(s)$ be defined by
\begin{align} \label{MMconstant}
A(s) = 
&\frac{10s - 12}{s-1} + s \sum_{l = 1}^{\infty}\frac{2(-1)^{l}(2l-2)!(2\pi)^{2l}}{\{(2l)!\}^2
\left(2l + s - 1 \right) \zeta(2l+1)}\\ \nonumber
&-s\sum_{\gamma}\frac{1}{\zeta'(\rho)\rho(\rho+1)(\rho-s+1)}.
\end{align}
Now, we obtain the following result, which gives an improvement on a 
result of Ng \cite{ng1} (see the remark after Corollary \ref{zetaequation}).


\begin{corollary}	\label{MM}
Assume the weak Mertens Hypothesis and let $\kappa$ be a real number. Then, we have
\begin{align*}
\int_{1}^{x}\frac{M(u)}{u^{\kappa}}du
= x^{3/2-\kappa}\sum_{\gamma}\frac{x^{i\gamma}}{\zeta'(\rho)\rho(\rho+1-\kappa)} + E_{\kappa}(x),
\end{align*}
where
\begin{align*}
E_{\kappa}(x)
=
\l\{\begin{array}{ll}
	A(\kappa) + O\l(x^{1 - \kappa}\r) 	& \text{\normal{if} \; $\d{1 < \kappa}$}, \vspace{1mm} \\
	O(\log{x}) 					& \text{\normal{if} \; $\kappa = 1$}, \vspace{1mm} \\
	O(x^{1-\kappa}) 				& \text{\normal{if} \; $\kappa < 1$}.
\end{array}
\r.
\end{align*}
Moreover, under the weak Mertens Hypothesis, for any $\kappa > \frac{1}{2}$ we have
\begin{align}	\label{zetaequationreal}
\frac{1}{\zeta(\kappa)} = \kappa A(\kappa + 1),
\end{align}
and for any $\kappa \leq \frac{3}{2}$
\begin{align}	\label{impNg}
\int_{1}^{x}\frac{M(u)}{u^{\kappa}}du \ll x^{3/2 - \kappa}.
\end{align}
\end{corollary}

We were able to obtain the equation (\ref{zetaequationreal}), which express $1/\zeta(s)$
 by the sum over the non-trivial zeros of $\zeta(s)$ under the weak Mertens Hypothesis. 
Also, this equation holds even if $\kappa$ is any complex number by analytic continuation, or we obtain the following corollary as
an explicit formula of $1/\zeta(s)$ on $\mathbb{C}$.


\begin{corollary}	\label{zetaequation}
Assume the weak Mertens Hypothesis. We have
\begin{align*}
\frac{1}{\zeta(s)} = 
&10s - 2 + s(s+1) \sum_{l = 1}^{\infty}\frac{2(-1)^{l}(2l-2)!(2\pi)^{2l}}{\{(2l)!\}^2
\left(2l + s \right) \zeta(2l+1)}\\ \nonumber
&-s(s+1)\sum_{\gamma}\frac{1}{\zeta'(\rho)\rho(\rho+1)(\rho-s)}
\end{align*}
for $s \in \mathbb{C}$, and these two series are absolutely and uniformly convergent in any compact subset $K \subset \mathbb{C}$,
which does not contain a zero of $\zeta(s)$.
\end{corollary}


Thanks to this result, we can represent the special values $\zeta(2n+1)$ by the sum over the non-trivial zeros of $\zeta(s)$
and the other special values $\zeta(2m+1)$.
For example, 
\begin{align*}
\frac{1}{\zeta(3)} = &
28 + 12\sum_{l = 1}^{\infty}\frac{2(-1)^{l}(2l-2)!(2\pi)^{2l}}{\{(2l)!\}^2(2l + 3) \zeta(2l+1)}
- 12 \sum_{\gamma}\frac{1}{\zeta'(\rho)\rho(\rho+1)(\rho-3)}, \\
\frac{1}{\zeta(5)} = &
48 + 30\sum_{l = 1}^{\infty}\frac{2(-1)^{l}(2l-2)!(2\pi)^{2l}}{\{(2l)!\}^2\left(2l + 5 \right) \zeta(2l+1)}
- 30 \sum_{\gamma}\frac{1}{\zeta'(\rho)\rho(\rho+1)(\rho-5)}
\end{align*}
hold under the weak Mertens Hypothesis. 


Here we remark on the upper bound of \eqref{impNg}.
We can easily find that
\begin{align}	\label{simpleconsideration}
\int_{1}^{x}\frac{M(u)}{u^{\kappa}}du
&\leq \left( \int_{1}^{x}\left(\frac{M(u)}{u}\right)^2du \right)^{1/2}
\left( \int_{1}^{x}\frac{du}{u^{2\kappa-2}} \right)^{1/2} \nonumber\\ 
&\ll \l\{\begin{array}{ll}
	\log{x} 				& \text{if \; $\kappa = 3/2$}, \vspace{1mm} \\
	x^{3/2-\kappa}(\log{x})^{1/2} & \text{if \; $\kappa < 3/2$}
\end{array}
\r.
\end{align}
as a simple consequence of the weak Mertens Hypothesis.
If $\kappa = 3/2$, a better estimate
\begin{align*}
\int_{1}^{x}\frac{M(u)}{u^{3/2}}du = o(\log{x})
\end{align*}
than \eqref{simpleconsideration}, which is equivalent to the formula (19) in \cite{ng1}, is known under the Riemann Hypothesis and $J_{-1}(T) \ll T$.
We have succeeded in obtaining the sharper estimate (\ref{impNg}) by calculating the explicit formula for $M_{\tau}(x)$.
In addition, we can find that the estimate (\ref{impNg}) is the best possible upper bound by the following theorem.


\begin{theorem} \label{divIM}
For any $\kappa \leq 3/2$, 
\begin{align*}
\int_{1}^{x}\frac{M(u)}{u^{\kappa}}du - \frac{2}{\zeta(1/2)}
= \Omega_{\pm}\l( x^{3/2-\kappa} \r)
\end{align*}
holds unconditionally, and the term $\frac{2}{\zeta(1/2)}$ is missing unless $\kappa = 3/2$.
In particular, we have
\begin{align}	\label{absIM}
\int_{1}^{\infty}\frac{|M(x)|}{x^{3/2}}dx = \infty.
\end{align}
\end{theorem}


\begin{remark}	\label{rem1}
We can show (\ref{absIM}) more easily since the Riemann zeta-function has a zero 
in the vertical strip $\frac{1}{2} \leq \sigma \leq 1$.
\end{remark}

On the other hand, we have
\begin{align}	\label{IMub}
\int_{1}^{x}\frac{|M(u)|}{u^{3/2}}du \ll \log{x}
\end{align}
under the weak Mertens Hypothesis by (\ref{simpleconsideration}).




Ng \cite{ng1} obtained an estimate for the logarithmic density 
\begin{align*}
\delta(S) := \lim_{X \rightarrow \infty}\frac{1}{\log{X}}\int_{[2, X] \cap S}\frac{dt}{t}
\end{align*}
of the set $S = \{ x \geq 1 \mid |M(x)| \leq \sqrt{x} \}$.
He showed that 
\begin{align}	\label{deltarange}
0 < \delta(S) < 1
\end{align}
under the Riemann Hypothesis, the Linear Independence Conjecture,
\begin{align*}
\sum_{0 < \gamma < T}\frac{1}{|\rho\zeta'(\rho)|} \asymp (\log{T})^{5/4} 
\quad \text{and} \quad \sum_{\gamma > T}\frac{1}{|\rho\zeta'(\rho)|^2} \asymp \frac{1}{T}.
\end{align*}
It was revealed by Gonek \cite{G} that the Riemann Hypothesis, $J_{-1}(T) \ll T$ and \eqref{Gonek-estimate} imply
\begin{align*}
\sum_{\gamma > T}\frac{1}{|\rho\zeta'(\rho)|^2} \asymp \frac{1}{T}
\end{align*}
so that we can find that the Riemann Hypothesis, the Linear Independence Conjecture,
$
\sum_{0 < \gamma < T}\frac{1}{|\rho\zeta'(\rho)|} \asymp (\log{T})^{5/4} 
\quad \text{and} \quad J_{-1}(T) \ll T 
$
imply
\begin{align*}
\int_{1}^{x}\frac{|M(u)|}{u^{3/2}}du \asymp \log{x}.
\end{align*}
Moreover, by \eqref{SWMH}, we have
\begin{align*}
\limsup_{x \rightarrow \infty}\frac{1}{\log{x}}\int_{1}^{x}\frac{|M(u)|}{u^{3/2}}du 
&\leq \left( \sum_{\gamma}\frac{1}{|\rho\zeta'(\rho)|^2} \right)^{1/2},\\
\liminf_{x \rightarrow \infty}\frac{1}{\log{x}}\int_{1}^{x}\frac{|M(u)|}{u^{3/2}}du 
& \geq 1 - \delta(S)
\end{align*}
under the same assumptions.

Furthermore, we consider the integral of the summatory function of $\mu(n)$ under the Linear Independence Conjecture and the weak Mertens Hypothesis.
In fact, we obtain the following theorem.


\begin{theorem} \label{IM}
Assume the Linear Independence Conjecture and let $\kappa \leq 3/2$. Then we have
\begin{align}
\limsup_{x \rightarrow \infty}\frac{1}{x^{3/2-\kappa}}\int_{1}^{x}\frac{M(u)}{u^{\kappa}}du 
&\geq \frac{2}{\zeta(1/2)} + \frac{1}{2}\sum_{\gamma}\frac{1}{|\rho(\rho - \kappa + 1)\zeta'(\rho)|} \label{IM1},\\
\liminf_{x \rightarrow \infty}\frac{1}{x^{3/2-\kappa}}\int_{1}^{x}\frac{M(u)}{u^{\kappa}}du 
&\leq \frac{2}{\zeta(1/2)} - \frac{1}{2}\sum_{\gamma}\frac{1}{|\rho(\rho - \kappa + 1)\zeta'(\rho)|}, \label{IM2}
\end{align}
where $\rho$ is a multiple zero, we understand that $\displaystyle{\frac{1}{|\rho(\rho - \kappa + 1)\zeta'(\rho)|}} = +\infty$, 
and the term $\frac{2}{\zeta(1/2)}$ is missing unless $\kappa = 3/2$.
\end{theorem}



\begin{corollary} \label{formula}

Assume the Linear Independence Conjecture, the weak Mertens Hypothesis and let $\kappa \leq 3/2$. 
Then we have
\begin{align*}
\frac{2}{\zeta(1/2)} + \frac{1}{2}\sum_{\gamma}\frac{1}{|\rho(\rho - \kappa + 1)\zeta'(\rho)|}
&\leq \limsup_{x \rightarrow \infty}\frac{1}{x^{3/2-\kappa}}\int_{1}^{x}\frac{M(u)}{u^{\kappa}}du\\
&\leq A(3/2) + \sum_{\gamma}\frac{1}{|\rho(\rho - \kappa + 1)\zeta'(\rho)|},
\end{align*}
and
\begin{align*}
A(3/2) - \sum_{\gamma}\frac{1}{|\rho(\rho - \kappa + 1)\zeta'(\rho)|}
&\leq \liminf_{x \rightarrow \infty}\frac{1}{x^{3/2-\kappa}}\int_{1}^{x}\frac{M(u)}{u^{\kappa}}du\\
&\leq \frac{2}{\zeta(1/2)} - \frac{1}{2}\sum_{\gamma}\frac{1}{|\rho(\rho - \kappa + 1)\zeta'(\rho)|}, 
\end{align*}
where $A(3/2)$ is defined by the equation (\ref{MMconstant}), 
and the terms $\frac{2}{\zeta(1/2)}$ and $A(3/2)$ are missing unless $\kappa = 3/2$.
\end{corollary}


We can easily prove Corollary \ref{formula} by using Theorem \ref{IM} and Corollary \ref{MM}.

\begin{remark}
Corollaries \ref{MM}, \ref{zetaequation} and \ref{formula} hold 
even if the weak Mertens Hypothesis is replaced by the assumptions
the Riemann Hypothesis, (SZC) and 
\begin{align*}
J_{-1}(T) := \sum_{0 < \gamma < T}\frac{1}{|\zeta'(\rho)|^2} \ll \frac{T^{3}}{\left(\log{T}\right)^{3+\delta}}
\end{align*}
for any fixed number $\delta > 0$.
\end{remark}


                                                   \section{\bf Auxiliary Lemmas} 


Let $s=\sigma+it$  be a complex variable with $\sigma$ and $t$ being real.


\begin{lemma} \label{zlb}
Let $T > 0$ be a sufficiently large positive number and $H = T^{1/3}$.
Then we have
\begin{align*}
\min_{T \leq t \leq T+H}\max_{\frac{1}{2} \leq \sigma \leq 2}|\zeta(\sigma + it)|^{-1} 
\leq \exp\left(C(\log{\log{T}})^2\right)
\end{align*}
with an absolute constant $C > 0$.
In particular, there exists a real number $T_* \in [T, T + T^{1/3}]$ such that 
\begin{align*}
\frac{1}{\zeta(\sigma + iT_*)} \ll {T_*}^\varepsilon \qquad \qquad \l(\frac{1}{2} \leq \sigma \leq 2\r)
\end{align*}
for any $\varepsilon > 0$.
\end{lemma}


\begin{proof}
The proof of this lemma is given in \cite[Theorem 2]{RS}.
\end{proof}


\begin{lemma} \label{zlbuRH}
Assume the Riemann Hypothesis. There exists an absolute constant $C > 0$ such that
\begin{align*}
\frac{1}{|\zeta(s)|} 
\leq \exp\left( \frac{C\log{t}}{\log{\log{t}}}\log\left( \frac{e}{(\sigma - 1/2)\log{\log{t}}} \right) \right),
\quad \left(\frac{1}{2} < \sigma \leq \frac{1}{2} + \frac{1}{\log{\log{t}}}\right)
\end{align*}
for $|t| \geq 5$.
\end{lemma}


\begin{proof}
This lemma is given in \cite[Theorem 13.\ 23]{MV}.
\end{proof}


\begin{lemma}	\label{zaf}
For any sufficiently large $t > T_0$ and $\sigma < \frac{1}{2}$, we have
\begin{align}	\label{afeq1}
\zeta(s) \asymp |s|^{1/2 - \sigma}(2\pi e)^{\sigma}\exp\l( t\tan^{-1}\l( \frac{1-\sigma}{t} \r) \r)\zeta(1-s).
\end{align}
If $t \leq T_0$ and $\sigma = -(2m+1) / 2$ with $m \in \mathbb{Z}_{>0}$, we have
\begin{align}	\label{afeq2}
\zeta(s) \asymp |\sigma|^{1/2}\l( \frac{2\pi e}{|\sigma|} \r)^{\sigma}\zeta(1-s).
\end{align} 
\end{lemma}


\begin{proof}
This lemma follows from the functional equation $\zeta(s) = \chi(s)\zeta(1-s)$ with $\chi(s)=2(2\pi)^{s-1}\Gamma(1-s)\sin\l( \frac{\pi s}{2} \r)$ 
and the Stirling formula.
\end{proof}


\begin{lemma}	\label{GRF}
Let $\tau > 0$ and $x > 0$ be real numbers, 
 $a(n)$ be an arithmetical function, $\a(s) = \sum_{n = 1}^{\infty}\frac{a(n)}{n^{s}}$ and
$\sigma_0 > \max\{0, \sigma_a\}$ with the abscissa of absolute convergence $\sigma_a$ for Dirichlet series $\alpha(s)$. 
Define:
\begin{align*}
C_\tau(x) &:= \frac{1}{\Gamma(\tau+1)}\sum_{n \leq x}a(n)\l( 1 - \frac{n}{x} \r)^\tau.
\end{align*}
Then, for any sufficiently large number $T > 0$, we obtain
\begin{align*}
C_\tau(x)
&= \frac{1}{2\pi i}\int_{\sigma_0-iT}^{\sigma_0+iT}\alpha(s)\frac{x^s \Gamma(s)}{\Gamma(1+\tau+s)}ds + R,
\end{align*}
where
\begin{align} \label{Rinequality}
R 
\ll &\frac{x^{\sigma_0}}{T^{1+\tau}}\sum_{n=1}^{\infty}\frac{|a(n)|}{n^{\sigma_0}}
+ \frac{x^{\sigma_0}}{T^{\tau}} \sum_{x/2 < n \leq x}\frac{|a(n)|}{n^{\s_{0}}}\min\l\{ 1 + \frac{(x/n)^{-(T-\s_0)}}{\tau}, \frac{x}{T|x-n|} \r\}\\ \nonumber
&+\frac{x^{\sigma_0}}{T^{\tau}}\sum_{x < n < 2x}\frac{|a(n)|}{n^{\s_{0}}}\min\l\{ 1 + \frac{(x/n)^{T-\s_0}}{\tau}, \frac{x}{T|x-n|} \r\}.
\end{align}
\end{lemma}


\begin{proof}

When $\max\{ \sigma, |t| \}$ is sufficiently large and $\sigma$ is non-negative real number, then we have
\begin{align}	\label{St1}
\frac{\Gamma(s)}{\Gamma(1+\tau+s)}
&\ll (\sigma^2 + t^2)^{-(1+\tau)/2}
\end{align}
by the Stirling formula. 
By the Stirling formula and $\Gamma(z)\Gamma(1-z) = \pi  / \sin \pi z$,
we also have \eqref{St1} when $t$ is any real number and $\sigma = -(2n+1) / 2$ is negative half-integer
with $n \in \mathbb{Z}_{>0}$, or $t > 0$ is any sufficiently large and $\sigma$ is any negative number.


Now, we are going to give the following estimate:
\begin{align}	\label{EQ1}
&\frac{1}{2\pi i}\int_{\sigma_0-iT}^{\sigma_0+iT}y^s\frac{\Gamma(s)}{\Gamma(1+\tau+s)}ds\\ \nonumber
&= 
\begin{dcases}
	O\l( y^{\sigma_0} / T^{1+\tau} \r)	&	\text{if \; $0 < y \leq \frac{1}{2}$,} \\
	O\l(\min\left\{\frac{y^{\sigma_0}}{T^{\tau}} + \frac{y^{T}}{\tau T^{\tau}}, \frac{y^{\sigma_0}T^{-(1+\tau)}}{|\log{y}|}\right\}\r) 
			& \text{if \; $\frac{1}{2} \leq y \leq 1$, } \\
	\frac{(1 - y^{-1})^\tau}{\Gamma(1+\tau)} 
	+ O\l(\min\left\{\frac{y^{\sigma_0}}{T^{\tau}} + \frac{y^{-T}}{\tau T^{\tau}}, \frac{y^{\sigma_0}T^{-(1+\tau)}}{|\log{y}|}\right\}\r)  
			& \text{if \; $1 \leq y \leq 2$,}\\
	\frac{(1 - y^{-1})^\tau}{\Gamma(1+\tau)} + O\l( y^{\sigma_0} / T^{1+\tau} \r)	&	\text{if \; $2 \leq y$}.
\end{dcases}
\end{align}
First we consider the case $0 < y \leq 1$. 
Let $K > T$ be a sufficiently large number.
By Cauchy's theorem, we find that
\begin{align*}
\frac{1}{2\pi i}\int_{\sigma_0-iT}^{\sigma_0+iT}y^s\frac{\Gamma(s)}{\Gamma(1+\tau+s)}ds
&=\frac{1}{2\pi i}\left(\int_{K + iT}^{\sigma_0+iT}+\int_{K-iT}^{K+iT}+\int_{\sigma_0-iT}^{K-iT}\right)y^s\frac{\Gamma(s)}{\Gamma(1+\tau+s)}ds\\
&=: I_1 + I_2 + I_3,
\end{align*}
say. We get 
\begin{align*}
|I_2|
&\ll y^K\int_{-T}^{T}\l( K^2 + t^2 \r)^{-(1+\tau)/2}dt 
\ll y^K K^{-1 - \tau}T.
\end{align*}
When $K \rightarrow \infty$, the last term goes to zero. For $0 < y < 1$, we also have
\begin{align*}
|I_1|
&\ll \int_{\sigma_0}^{K}y^{\sigma}\l( \sigma^2 + T^2 \r)^{-(1+\tau)/2}d\sigma   \nonumber
\ll T^{-(1+\tau)}\int_{\sigma_0}^{K}y^\sigma d\sigma 
\ll \frac{y^{\sigma_0}T^{-(1+\tau)}}{\log(1/y)}.\\ \nonumber
\end{align*}
On the other hand, for $0 < y \leq 1$ we have
\begin{align*}
|I_1|
&\ll \left(\int_{\sigma_0}^{T} + \int_{T}^{K} \right)y^{\sigma}\l( \sigma^2 + T^2 \r)^{-(1+\tau)/2}d\sigma \\
&\ll y^{\sigma_{0}}\int_{\sigma_0}^{T}T^{-1-\tau} d\sigma + y^{T}\int_{T}^{K}\sigma^{-1-\tau}d\sigma
\ll \frac{y^{\sigma_{0}}}{T^{\tau}} + \frac{y^{T}}{\tau T^{\tau}}.
\end{align*}
By the Schwarz reflection principle, we see that $|I_{1}| = |I_3|$.

Therefore, for $0 < y \leq 1$, we get
\begin{align}	\label{IEQP1}
\frac{1}{2\pi i}\int_{\sigma_0-iT}^{\sigma_0+iT}\frac{y^s\Gamma(s)}{\Gamma(1+\tau+s)}ds 
\ll \min\left\{\frac{y^{\sigma_{0}}}{T^{\tau}} + \frac{y^{T}}{\tau T^{\tau}}, \frac{y^{\sigma_0}T^{-(1+\tau)}}{|\log{y}|}\right\}.
\end{align}
In particular, for $0 < y \leq 1/2$, we have
\begin{align*}
\frac{1}{2\pi i}\int_{\sigma_0-iT}^{\sigma_0+iT}\frac{y^s\Gamma(s)}{\Gamma(1+\tau+s)}ds
\ll y^{\sigma_0}/T^{1+\tau}.
\end{align*}

Next, we consider the case $1 \leq y$. Let $L > T$ be a sufficiently large positive half integer. 
By the residue theorem, we find that
\begin{align*}
\frac{1}{2\pi i}\int_{\sigma_0-iT}^{\sigma_0+iT}\frac{y^s\Gamma(s)}{\Gamma(1+\tau+s)}ds
=&\frac{1}{2\pi i}\left(\int_{-L + iT}^{\sigma_0+iT}+\int_{-L-iT}^{-L+iT}+\int_{\sigma_0-iT}^{-L-iT}\right)\frac{y^s\Gamma(s)}{\Gamma(1+\tau+s)}ds\\
&+\sum_{0 \leq n < L}\underset{s=-n}{\mathrm{Res}}\left( y^s\frac{\Gamma(s)}{\Gamma(1+\tau+s)} \right)\\
=&: I_4 + I_5 + I_6 + \sum_{0 \leq n < L}\underset{s=-n}{\mathrm{Res}}\left( y^s\frac{\Gamma(s)}{\Gamma(1+\tau+s)} \right), 
\end{align*}
say. We notice that 
\begin{align*}
|I_5|
&\ll y^{-L}\int_{-T}^{T}\l(L^2 + t^2\r)^{-(1+\tau)/2}dt
\ll y^{-L} L^{-1 - \tau}T.
\end{align*}
When $L \rightarrow \infty$, the last term goes to zero.
For $y > 1$, we also get
\begin{align*}
|I_4|
&\ll \int_{-L}^{\sigma_0}y^{\sigma}\l( \sigma^2 + T^2 \r)^{-(1+\tau)/2}d\sigma
\ll T^{-(1+\tau)}\int_{-L}^{\sigma_0}y^\sigma d\sigma 
\ll \frac{y^{\sigma_0}T^{-(1+\tau)}}{\log{y}}.
\end{align*}
On the other hand, for $y \geq 1$, we have
\begin{align*}
|I_4|
&\ll \left(\int_{-T}^{\sigma_0} + \int_{-L}^{-T} \right)y^{\sigma}\l( \sigma^2 + T^2 \r)^{-(1+\tau)/2}d\sigma\\
&\ll y^{\sigma_{0}}\int_{-T}^{\sigma_0}T^{-1-\tau} d\sigma + y^{-T}\int_{-L}^{-T}|\sigma|^{-1-\tau}d\sigma
\ll \frac{y^{\s_{0}}}{T^{\tau}} + \frac{y^{-T}}{\tau T^{\tau}}.  
\end{align*}
By the Schwarz reflection principle, we see that $|I_4| = |I_6|$.

Therefore, we have
\begin{align}	\label{IEQP2}
&\frac{1}{2\pi i}\int_{\sigma_0-iT}^{\sigma_0+iT}y^s\frac{\Gamma(s)}{\Gamma(1+\tau+s)}ds \\ \nonumber
&= \sum_{n=0}^{\infty}\underset{s=-n}{\mathrm{Res}}\left( y^s\frac{\Gamma(s)}{\Gamma(1+\tau+s)} \right)
+ O\left(\min\left\{\frac{y^{\s_{0}}}{T^{\tau}} + \frac{y^{-T}}{\tau T^{\tau}}, \frac{y^{\sigma_0}T^{-(1+\tau)}}{\log{y}}\right\}\right).
\end{align}
Hence we obtain the inequality (\ref{EQ1}) by (\ref{IEQP1}), (\ref{IEQP2}) and
\begin{align*}
&\sum_{n=0}^{\infty}\underset{s=-n}{\mathrm{Res}}\left( y^s\frac{\Gamma(s)}{\Gamma(1+\tau+s)} \right)
= \sum_{n=0}^{\infty}\l(-y^{-1}\r)^{n}\frac{1}{\Gamma(1+\tau-n)n!}
= \frac{\l(1 - y^{-1}\r)^\tau}{\Gamma(1+\tau)}.
\end{align*}
From the above argument, we obtain the estimate
\begin{align*}
&\frac{1}{2\pi i}\int_{\sigma_0 - iT}^{\sigma_0 + iT}\alpha(s)x^s\frac{\Gamma(s)}{\Gamma(1+\tau+s)}ds\\
&= \frac{1}{2\pi i}\left(\sum_{n \leq x} +\sum_{{n > x}} \right)a(n)
\int_{\sigma_0 - iT}^{\sigma_0+iT}\l(\frac{x}{n}\r)^s\frac{\Gamma(s)}{\Gamma(1+\tau+s)}ds\\
&= \frac{1}{\Gamma(1+\tau)}\sum_{n \leq x}a(n)\l( 1 - \frac{n}{x} \r)^\tau - R, 
\end{align*}
where we can find that $R$ satisfies the inequality \eqref{Rinequality}.
\end{proof}


\begin{lemma}	\label{gzdbz}
Let $N(T)$ be the number of non-trivial zeros $\rho = \beta + i\gamma$ of the Riemann zeta-function with 
$0 < \gamma \leq T$ counted with multiplicity. Then we have
\begin{align*}
	N(T) = \frac{1}{2\pi}T\log{T} - \frac{1+\log{2\pi}}{2\pi}T + O\left( \log{T} \right).
\end{align*}
\end{lemma}

\begin{proof}
The proof is given in \cite[Corollary 14.\ 4]{MV}.
\end{proof}








                                          \section{\bf Proofs of Theorems \ref{SZCApuRH}, \ref{exfoM} and \ref{ASZC}}


First, we have to show the following Lemma. Theorem \ref{SZCApuRH} is immediate consequence of this lemma.

\begin{lemma}	\label{lemSZC}

Assume the Riemann Hypothesis.
Let $\rho = \frac{1}{2} + i\gamma$ be a non-trivial zero of the Riemann zeta-function, 
$m(\rho)$ be the multiplicity of $\rho$ and $\tau$ be a positive real number.
We have
\begin{align*}
M_{\tau}(x) = \Omega_{\pm}\left(x^{1/2}(\log{x})^{m(\rho)-1} \right).
\end{align*}
\end{lemma}

The author prove this lemma by using the idea in Section 15.1 in \cite{MV}.

\begin{proof}

Let $\tau$ be a positive number and $c > 0$ be an absolute constant such that 
$M_{\tau}(x) \leq cx^{1/2}(\log{x})^{m(\rho)-1}$ for any $x \geq K_0$.
We define
\begin{align*}
G_{\tau}(s) := \int_{1}^{\infty}\frac{M_{\tau}(x) - cx^{1/2}(\log{x})^{m(\rho)-1}}{x^{s+1}}dx.
\end{align*}
By Lemma \ref{GRF}, $M_{\tau}(x)$ is expressed by the Mellin inversion transform of $\frac{\Gamma(s)}{\zeta(s)\Gamma(1+\tau+s)}$ as follows:
\begin{align}	\label{mellin inversion M_{tau}}
M_{\tau}(x) = \frac{1}{2\pi i}\int_{\s_{0} - i\infty}^{\s_{0} + i\infty}\frac{x^s}{\zeta(s)}\frac{\Gamma(s)}{\Gamma(1+\tau+s)}ds
\end{align}
for $\s_{0} > 1$.
On the other hand, we can prove that the inequality $M_{\tau}(x) \ll x^{1/2 + \varepsilon}$ holds for any $\tau \geq 0$ 
under the Riemann Hypothesis by using Lemma \ref{GRF} in the same way as in the proof of 
$M(x) \ll x^{1/2 + \varepsilon}$, which is given in \cite[Theorem 14.25]{T}.
Hence, for $\sigma > 1/2$, we have
\begin{align} \label{G_{tau} formula}
G_{\tau}(s) = \frac{\Gamma(s)}{\zeta(s)\Gamma(1+\tau+s)} - \frac{c(m(\rho)-1)!}{(s-1/2)^{m(\rho)}}
\end{align}
by the assumption of the Riemann Hypothesis, equation (\ref{mellin inversion M_{tau}}) and the Mellin transform.

Next, we define
\begin{align*}
I_{\tau}(s) := \int_{1}^{\infty}\frac{M_{\tau}(x) - cx^{1/2}(\log{x})^{m(\rho)-1}}{x^{s+1}}(1 + \cos(\phi - \gamma\log{x}))dx,
\end{align*}
where $\phi$ is any real number.
Then we have
\begin{align} \label{I_{tau} formula}
I_{\tau}(s)
 &=G_{\tau}(s) + \frac{1}{2}\left(e^{i\phi}G_{\tau}(s+i\gamma) + e^{-i\phi}G_{\tau}(s-i\gamma) \right)
\end{align}
from the equation $\cos(\phi - i \gamma \log{x}) = \frac{1}{2}\l(e^{i\phi}x^{-i\gamma} + e^{-i\phi}x^{i\gamma}\r)$.

We can find that $I_{\tau}(\sigma)$ does not tend to $+\infty$ as $\sigma \rightarrow 1/2 + 0$ for any real $\phi$ 
from the above integral representation of $I_{\tau}(s)$.
Hence we have
\begin{align*}
\limsup_{\sigma \rightarrow 1/2 + 0}(\sigma - 1/2)^{m(\rho)}I_{\tau}(\sigma) \leq 0.
\end{align*}

On the other hand, by \eqref{G_{tau} formula} and \eqref{I_{tau} formula}, we have
\begin{align*}
\lim_{\sigma \rightarrow 1/2+0}(\sigma - 1/2)^{m(\rho)}I_{\tau}(\s)
= -c(m(\rho)-1)! + \Re\left(e^{i\phi}\frac{\Gamma(\rho)(m(\rho))!}{\Gamma(1+\tau+\rho)\zeta^{(m(\rho))}(\rho)}\right).
\end{align*}
Thus we have
\begin{align*}
0 < \frac{m(\rho)|\Gamma(\rho)|}{\left|\Gamma(1+\tau+\rho)\zeta^{(m(\rho))}(\rho)\right|} \leq c,
\end{align*}
if we take $\phi = \arg\left(\frac{\Gamma(1+\tau+\rho)\zeta^{(m(\rho))}(\rho)}{\Gamma(\rho)}\right)$.
Hence we have
\begin{align*}
M_{\tau}(x) = \Omega_{+}\left(x^{1/2}(\log{x})^{m(\rho)-1} \right).
\end{align*}
Similarly, we have
\begin{align*}
M_{\tau}(x) = \Omega_{-}\left(x^{1/2}(\log{x})^{m(\rho)-1} \right),
\end{align*}
which completes the proof of Lemma \ref{lemSZC}.
\end{proof}



\begin{proof}[Proof of Theorem \ref{exfoM}]
Let $\tau$ be a positive number and $L = (2m+1) / 2$ with $m \in \mathbb{Z}_{>0}$.
By Lemma \ref{GRF}, we have
\begin{align*}
M_\tau(x) 	&= \frac{1}{2\pi i}\int_{2 -iT_*}^{2 + iT_*}\frac{x^s}{\zeta(s)}\frac{\Gamma(s)}{\Gamma(1+\tau+s)}ds 
				+ O\l(\frac{x^{2}}{\tau {T_*}^{\tau}} \r),
\end{align*}
where $T_*$ satisfies the condition in Lemma \ref{zlb}. By the residue theorem, we find that
\begin{align*}	
&\frac{1}{2\pi i}\int_{2 -iT_*}^{2 + iT_*}\frac{x^s}{\zeta(s)}\frac{\Gamma(s)}{\Gamma(1+\tau+s)}ds \\
&= \frac{1}{2\pi i}\l(\int_{-L+iT}^{2 + iT_*}+\int_{-L-iT_*}^{-L + iT_*}+\int_{2-iT_*}^{-L - iT_*}\r)
\frac{x^s}{\zeta(s)}\frac{\Gamma(s)}{\Gamma(1+\tau+s)}ds	\\ 
&	+\sum_{|\gamma|<T_*}\underset{s=\rho}{\rm Res}\l(\frac{x^s}{\zeta(s)}\frac{\Gamma(s)}{\Gamma(1+\tau+s)}\r)
	+\sum_{0 \leq l < L}\underset{s=-l}{\rm Res}\l(\frac{x^s}{\zeta(s)}\frac{\Gamma(s)}{\Gamma(1+\tau+s)}\r)\\
&=: J_1 + J_2 + J_3
+\sum_{|\gamma|<T_*}\underset{s=\rho}{\rm Res}\l(\frac{x^s}{\zeta(s)}\frac{\Gamma(s)}{\Gamma(1+\tau+s)}\r)
	+\sum_{0 \leq l < L}\underset{s=-l}{\rm Res}\l(\frac{x^s}{\zeta(s)}\frac{\Gamma(s)}{\Gamma(1+\tau+s)}\r),
\end{align*}
say.
We have
\begin{align*}
&\underset{s=\rho}{\rm Res}\l(\frac{x^s}{\zeta(s)}\frac{\Gamma(s)}{\Gamma(1+\tau+s)}\r)
= \frac{1}{(m(\rho)-1)!}\lim_{s \rightarrow \rho}\frac{d^{m(\rho)-1}}{ds^{m(\rho)-1}}
\left((s-\rho)^{m(\rho)}\frac{x^{s}}{\zeta(s)}\frac{\Gamma(s)}{\Gamma(1+\tau+s)}\right).
\end{align*}
We also have
\begin{align*}
\underset{s=0}{\rm Res}\l(\frac{x^s}{\zeta(s)}\frac{\Gamma(s)}{\Gamma(1+\tau+s)}\r) = -\frac{2}{\Gamma(1+\tau)}
\end{align*}
by $\zeta(0) = -1/2$. 
We will calculate the other residues in the next section, from which it follows that the series 
\begin{align*}
\sum_{l=1}^{\infty}\underset{s=-l}{\rm Res}\l(\frac{x^s}{\zeta(s)}\frac{\Gamma(s)}{\Gamma(1+\tau+s)}\r)
\end{align*}
is absolutely and uniformly convergent.

Next, we evaluate the integral. 
First, we estimate the integral along the vertical line. 
We notice that
\begin{align*}
|J_2|
\leq 	&\l|\int_{|t| \leq T_0}\frac{x^{-L+it}}{\zeta(-L+it)}\frac{\Gamma(-L+it)}{\Gamma(1+\tau-L+it)}dt \r|\\
	&+\l|\int_{T_0< |t| \leq T_*}\frac{x^{-L+it}}{\zeta(-L+it)}\frac{\Gamma(-L+it)}{\Gamma(1+\tau-L+it)}dt \r|,
\end{align*}
where $T_0 > 1$ is a sufficiently large constant. 
By the inequality \eqref{St1} and the asymptotic formula \eqref{afeq2} in Lemma \ref{zaf}, we have
\begin{align*}
\l|\int_{|t| \leq T_0}\frac{x^{-L+it}}{\zeta(-L+it)}\frac{\Gamma(-L+it)}{\Gamma(1+\tau-L+it)}dt \r|
&\ll \frac{(x / 2\pi e)^{-L}}{L^{L+3/2+\tau}}.
\end{align*}
The right-hand side tends to zero as $L \rightarrow \infty$.
By the inequality \eqref{St1} and the asymptotic formula \eqref{afeq1} in Lemma \ref{zaf}, we also have
\begin{align*}
\l|\int_{T_0< |t| \leq T_*}\frac{x^{-L+it}}{\zeta(-L+it)}\frac{\Gamma(-L+it)}{\Gamma(1+\tau-L+it)}dt \r|
&\ll \frac{(x / 2\pi e)^{-L}}{L^{1+\tau}}\int_{T_0< |t| \leq T_*}t^{-1/2 - L}dt\\
&\ll \frac{(x / 2\pi e)^{-L}}{L^{2+\tau}T_0^{L - 1/2}}.
\end{align*}
The last term tends to zero as $L \rightarrow \infty$.
Hence we obtain
\begin{align*}	
\lim_{L \rightarrow \infty}J_2 = 0.
\end{align*} 

Next, we evaluate the integral along one of the horizontal lines. One notes that
\begin{align*}
J_1
&=\left(\int_{1/2 + iT_*}^{2+iT_*} + \int_{iT_*}^{1/2+iT_*} + \int_{-L+iT_*}^{iT_*}\right)\frac{x^s}{\zeta(s)}\frac{\Gamma(s)}{\Gamma(1+\tau+s)}ds\\
&=: {J_1}' + {J_1}'' + {J_1}'''.
\end{align*}
By Lemma \ref{zlb} and \eqref{St1}, we have
\begin{align*}
|{J_1}'|
&\ll \int_{1/2}^2 x^\sigma {T_*}^{-1-\tau+\varepsilon}d\s \ll \frac{x^2 {T_*}^{-1-\tau+\varepsilon}}{\log{x}}.
\end{align*}
Using Lemma \ref{zlb}, (\ref{afeq1}) and (\ref{St1}), we find that
\begin{align*}
|{J_1}''|
\ll \frac{1}{{T_*}^{3/2+\tau-\varepsilon}}\int_{0}^{1/2}(xT_*)^{\sigma}d\sigma
\ll \frac{x^{1/2}}{{T_*}^{1+\tau-\varepsilon}\log{x}}.
\end{align*}
By (\ref{afeq1}) and (\ref{St1}), one has
\begin{align*}
|{J_1}'''|
\ll \frac{1}{{T_*}^{3/2+\tau-\varepsilon}}\int_{-L}^{0}(xT_*)^{\sigma}d\sigma
\ll \frac{1 - (xT_*)^{-L}}{{T_*}^{3/2+\tau-\varepsilon}\log(xT_*)}
\ll \frac{1}{{T_*}^{3/2+\tau-\varepsilon}\log{xT_*}}.
\end{align*} 
Hence we have
\begin{align*}	
J_1 
\ll \frac{x^2{T_*}^{-1-\tau+\varepsilon}}{\log{x}}.
\end{align*}
Similarly, we have
\begin{align*}
J_3 
\ll \frac{x^2{T_*}^{-1-\tau+\varepsilon}}{\log{x}}
\end{align*}
by the Schwarz reflection principle.
From the above argument, we obtain
\begin{align}	\label{subexfoM}
M_\tau(x) = &
\sum_{|\gamma| < T_*}\frac{1}{(m(\rho)-1)!}\lim_{s \rightarrow \rho}\frac{d^{m(\rho)-1}}{ds^{m(\rho)-1}}
			\left((s-\rho)^{m(\rho)}\frac{x^{s}}{\zeta(s)}\frac{\Gamma(s)}{\Gamma(1+\tau+s)}\right)\\ \nonumber
&+\sum_{l=0}^{\infty}\underset{s=-l}{\rm Res}\l(\frac{x^s}{\zeta(s)}\frac{\Gamma(s)}{\Gamma(1+\tau+s)}\r) 
 + O\l( \frac{x^{2}}{\tau {T_*}^{\tau}} + \frac{x^2{T_*}^{-1-\tau+\varepsilon}}{\log{x}} \r).
\end{align}
By taking a sequence $\{T_\nu\}$ which satisfies the condition in Lemma \ref{zlb} and $T_\nu \rightarrow \infty$, we find
\begin{align*}
M_\tau(x) =& 
\lim_{T_\nu \rightarrow \infty}\sum_{|\gamma| < T_{\nu}}\frac{1}{(m(\rho)-1)!}\lim_{s \rightarrow \rho}\frac{d^{m(\rho)-1}}{ds^{m(\rho)-1}}
			\left((s-\rho)^{m(\rho)}\frac{x^{s}}{\zeta(s)}\frac{\Gamma(s)}{\Gamma(1+\tau+s)}\right)\\
&+\sum_{l=0}^{\infty}\underset{s=-l}{\rm Res}\l(\frac{x^s}{\zeta(s)}\frac{\Gamma(s)}{\Gamma(1+\tau+s)}\r),
\end{align*}
which completes the proof of Theorem \ref{exfoM}.
\end{proof}

\begin{proof}[Proof of Theorem \ref{ASZC}]

We assume the Riemann Hypothesis.

Let $\sigma_0 = \frac{1}{2} + \frac{1}{\log{x}}$ and $\tau \geq \frac{2C\log{\log{T}}}{\log{\log{\log{\log{T}}}}} + 1$,
where $C$ satisfies the condition of Lemma \ref{zlbuRH}.
By Lemma \ref{GRF}, we have
\begin{align*}
M_\tau(x) 	&= \frac{1}{2\pi i}\int_{1+\frac{1}{\log{x}} -iT_*}^{1+\frac{1}{\log{x}} + iT_*}\frac{x^s}{\zeta(s)}\frac{\Gamma(s)}{\Gamma(1+\tau+s)}ds 
				+ O\l(\frac{x\log{x}}{\tau {T_*}^{\tau}} \r),
\end{align*}
where $T_* \in \left[T, T+T^{3/4}\right]$ satisfies the condition of Lemma \ref{zlb}. By the Cauchy theorem and the Riemann Hypothesis, we find that
\begin{align*}
&\frac{1}{2\pi i}\int_{1+1/\log{x} -iT_*}^{1+1/\log{x} + iT_*}\frac{x^s}{\zeta(s)}\frac{\Gamma(s)}{\Gamma(1+\tau+s)}ds\\
&= \frac{1}{2\pi i}\l(\int_{\sigma_0+iT_*}^{1+1/\log{x} + iT_*}+\int_{\sigma_0-iT_*}^{\sigma_0 + iT_*}
+\int_{1+1/\log{x}-iT_*}^{\sigma_0-iT_*}\r)\frac{x^s}{\zeta(s)}\frac{\Gamma(s)}{\Gamma(1+\tau+s)}ds\\
&=: L_1 + L_2 + L_3, 
\end{align*}
say. Using Lemma \ref{zlb} and inequality (\ref{St1}), we have
\begin{align*}
|L_1| = \left|\int_{\sigma_0+iT_*}^{1+1/\log{x} + iT_*}\frac{x^s}{\zeta(s)}\frac{\Gamma(s)}{\Gamma(1+\tau+s)}ds\right|
\ll x{T_*}^{\varepsilon-\tau}.
\end{align*}
Similarly, we have
\begin{align*}
|L_3| = \left|\int_{\sigma_0-iT_*}^{1+1/\log{x} - iT_*}\frac{x^s}{\zeta(s)}\frac{\Gamma(s)}{\Gamma(1+\tau+s)}ds\right|
\ll xT^{\varepsilon-\tau}.
\end{align*}
We split the following integral into three parts:
\begin{align*}
L_2 &= \left(\int_{|t| \leq 14} + \int_{14 < |t| \leq \log{\log{T}}} + \int_{\log{\log{T}} < |t| \leq T_*}\right)
\frac{x^{\sigma_0+it}}{\zeta(\sigma_0+it)}
\frac{\Gamma(\sigma_0+it)}{\Gamma(1+\tau+\sigma_0+it)}dt.
\end{align*}
We observe
\begin{align*}
\left|\int_{|t| \leq 14}\frac{x^{\sigma_0+it}}{\zeta(\sigma_0(T_*)+it)}\frac{\Gamma(\sigma_0(T_*)+it)}{\Gamma(1+\tau+\sigma_0(T_*)+it)}dt\right|
\ll x^{1/2}
\end{align*}
since the function $\d{\frac{\Gamma(s)}{\zeta(s)\Gamma(1+\tau+s)}}$ is continuous in the region $1/2 \leq \sigma$, $|t| \leq 14$ 
under the Riemann Hypothesis.
By Lemma \ref{zlbuRH} and the Stirling formula, we have
\begin{align*}
&\int_{14 < |t| \leq \log{\log{T}}}\left|\frac{x^{\sigma_0+it}}{\zeta(\sigma_0+it)}
\frac{\Gamma(\sigma_0+it)}{\Gamma(1+\tau+\sigma_0+it)}\right|dt\\
&\ll x^{1/2}\int_{14 < t < \log{\log{T}}}\exp\left( \frac{C\log{t}}{\log{\log{t}}}\log\left( \frac{e\log{x}}{\log{\log{t}}} \right) \right)\tau^{-\tau-1}dt\\
&\ll x^{1/2}\log{\log{T}} \exp\left( \frac{C\log{\log{x}}\log{\log{\log{T}}}}{\log{\log{\log{\log{T}}}}} - 
\frac{2C\log{\log{T}}\log{\log{\log{T}}}}{\log{\log{\log{\log{T}}}}} \right)\\
&\ll x^{1/2},
\end{align*}
if $T \geq x$ and 
\begin{align*}
&\int_{\log{\log{T}} < |t| \leq T_*}\left|\frac{x^{\sigma_0+it}}{\zeta(\sigma_0+it)}
\frac{\Gamma(\sigma_0+it)}{\Gamma(1+\tau+\sigma_0+it)}\right|dt\\
&\ll x^{1/2}\int_{\log{\log{T}} < t < 2T}
\exp\left( \frac{C\log{t}}{\log{\log{t}}}\log\left( \frac{e\log{2T}}{\log{\log{t}}} \right) \right)t^{-\tau-1}dt\\
&\ll x^{1/2}\int_{\log{\log{T}} < t < 2T}
\exp\left(\log{t}(\log{\log{T}})\left(\frac{C}{\log{\log{t}}} - \frac{2C}{\log{\log{\log{\log{T}}}}} \right) \right)\frac{dt}{t^2}\\
&\ll x^{1/2}.
\end{align*}
Therefore, we obtain
\begin{align*}
M_{\tau}(x) \ll x^{1/2} + \frac{x\log{x}}{\tau T^{\tau}} + xT^{\varepsilon-\tau}.
\end{align*}
Letting $T = x$, we have
\begin{align*}
M_{\tau}(x) \ll x^{1/2}
\end{align*}
for any $\tau \geq \frac{C_0\log{\log{x}}}{\log{\log{\log{\log{x}}}}}$.
\end{proof}


			\section{\bf Convergence of the residues series on negative integers} 


In this section we supply a proof of a fact, which was remained pending in the preceding section.
That is, we show that
\begin{align*}
\sum_{l=1}^{\infty}\underset{s = -l}{\rm Res}\left( \frac{x^{s}}{\zeta(s)}\frac{\Gamma(s)}{\Gamma(s+\tau+1)} \right)
\end{align*}
is absolutely and uniformly convergent.
First, we show the convergence in the case of non-integer $\tau$. 
When $l$ is an odd integer, we have
\begin{align*}
\underset{s=-2n+1}{\rm Res}\l(\frac{x^s}{\zeta(s)}\frac{\Gamma(s)}{\Gamma(1+\tau+s)}\r)
=-\frac{x^{-2n+1}}{\zeta(-2n+1)\Gamma(-2n+2+\tau)(2n-1)!},
\end{align*}
and
\begin{gather*}
\zeta(-2n+1)  		  = 2(-1)^n(2\pi)^{-2n}(2n-1)!\zeta(2n),\\
\Gamma(-2n+2+\tau) 	= \frac{1}{\sin(\pi \tau)\Gamma(2n-1-\tau)}.
\end{gather*}
Hence we have 
\begin{align*}
\sum_{n=1}^{\infty}\underset{s=-2n+1}{\rm Res}\l(\frac{x^s}{\zeta(s)}\frac{\Gamma(s)}{\Gamma(1+\tau+s)}\r) \ll x^{-1}
\end{align*}
and this series is absolutely and uniformly convergent.
When $n$ is even, we have
\begin{align*}
&\underset{s=-2n}{\rm Res}\l(\frac{x^s}{\zeta(s)}\frac{\Gamma(s)}{\Gamma(1+\tau+s)}\r)\\
&=\lim_{s \rightarrow -2n}\frac{d}{ds}\l( (s+2n)^2 \frac{x^s}{\zeta(s)}\frac{\Gamma(s)}{\Gamma(1+\tau+s)} \r)\\
&=\lim_{s \rightarrow -2n}\frac{(s+2n)x^s\Gamma(s)}{\zeta(s)\Gamma(1 + \tau + s)}\l( 2 - (s+2n)\frac{\zeta'(s)}{\zeta(s)} 
+(s+2n)\frac{\Gamma'(s)}{\Gamma(s)} \r)\\
& + \frac{x^{-2n}}{\zeta'(-2n)\Gamma(1+\tau-2n)(2n)!}\l(\log{x} - \frac{\Gamma'(1+\tau-2n)}{\Gamma(1+\tau-2n)} \r)\\
&= \frac{x^{-2n}}{\zeta'(-2n)\Gamma(1+\tau-2n)(2n)!}
\l( \log{x} - \frac{\zeta''(-2n)}{2\zeta'(-2n)} + c_{2n}(2n)! - \frac{\Gamma'(1+\tau-2n)}{\Gamma(1+\tau-2n)} \r),
\end{align*}
where $c_{2n}$ is the 0-th coefficient of the Laurent expansion of the gamma-function at $-2n$. 
Now, we find that
\begin{gather}
\zeta'(-2n)  = (-1)^n(2\pi)^{-2n}(2n)!\zeta(2n+1)/2, \label{zeta'}\\
\zeta''(-2n) = -2\zeta'(-2n)\left( \l(\frac{\pi}{2}\r)^2+\log{2\pi}-\frac{\Gamma'(2n+1)}{\Gamma(2n+1)}-\frac{\zeta'(2n+1)}{\zeta(2n+1)} \right) \label{zeta''}
\end{gather}
from 
$\d{\zeta(s) = \chi(s)\zeta(1-s)}$ and  
$\d{\zeta''(-2n) = 2\lim_{s \rightarrow -2n}\frac{(s+2n)\zeta'(s)-\zeta(s)}{(s+2n)^2}}$.\\
By (\ref{zeta'}) and (\ref{zeta''}), we have
\begin{align*}
\frac{\zeta''(-2n)}{\zeta'(-2n)}&=-2\l( \l(\frac{\pi}{2}\r)^2+\log{2\pi}-\frac{\Gamma'(2n+1)}{\Gamma(2n+1)}-\frac{\zeta'(2n+1)}{\zeta(2n+1)} \r).
\end{align*}
Using the functional equation $\Gamma(s)\Gamma(1-s) = \pi / \sin \pi s$, we find that
\begin{align*}
\frac{\Gamma'(1+\tau-2n)}{\Gamma(1+\tau-2n)}
= \frac{\Gamma(2n-\tau)}{\tan\pi(1+\tau)}-\Gamma'(2n-\tau).
\end{align*}
On the other hand, we have $c_{2n} = 0$ by 
\begin{align*}
c_{2n} = \lim_{s \rightarrow -2n}\l( \Gamma(s) - \frac{1}{(2n)!(s+2n)} \r) 
\quad \text{and} \quad
\Gamma(s)\Gamma(1-s) = \pi / \sin \pi s.
\end{align*}
Hence we obtain 
\begin{align*}
&\sum_{n=1}^{\infty}\underset{s=-2n}{\rm Res}\l(\frac{x^s}{\zeta(s)}\frac{\Gamma(s)}{\Gamma(1+\tau+s)}\r) 
\ll x^{-2}\log{x},
\end{align*}
and this series is absolutely and uniformly convergent.

Similarly, when $\tau = k$ is a positive integer, we obtain
\begin{align}	\label{Resint}
&\sum_{l=0}^{\infty}\underset{s=-l}{\rm Res}\l(\frac{x^s}{\zeta(s)}\frac{\Gamma(s)}{\Gamma(1+\tau+s)}\r)\\  \nonumber
&=\sum_{l > \frac{k}{2}}\frac{2(-1)^{l+k+1}(2l-k-1)!}{(2l!)^2\zeta(2l+1)}\l( \frac{x}{2\pi} \r)^{-2l}-\frac{2}{k!}\\ \nonumber
& + \sum_{0< l \leq \frac{k}{2}}\l\{ \frac{2\log{x}(x/2\pi)^{-2l}}{(2l!)^2(k-2l)!\zeta(2l+1)} 
		-\frac{2\zeta''(-2l)(x/4\pi^2)^{-2l}}{(2l!)^3(k-2l)!\l(\zeta(2l+1)\r)^2} \r\}\\ \nonumber
& -\sum_{0< l \leq \frac{k+1}{2}}\frac{x^{-2l+1}}{\zeta(-2l+1)(2l-1)!(k-2l+1)!}.
\end{align}
Therefore, the series
\begin{align*}
\sum_{l=1}^{\infty}\underset{s=-l}{\rm Res}\l(\frac{x^s}{\zeta(s)}\frac{\Gamma(s)}{\Gamma(1+\tau+s)}\r) \ll x^{-1}
\end{align*}
is absolutely and uniformly convergent for $\tau > 0$.


					\section{\bf Proofs of Corollaries \ref{dzb}, \ref{AWMH} and \ref{MM}} 


\begin{proof}[Proof of Corollary \ref{dzb}]

First, we show this corollary in the case $\tau \gg 1$.
Using equation (\ref{subexfoM}), we find that
\begin{align*}
M_\tau(x) = &
\sum_{|\gamma| < T_*}\frac{x^{i\gamma}}{\zeta'(\rho)}\frac{\Gamma(\rho)}{\Gamma(1+\tau+\rho)}x^{1/2}
+\sum_{l=0}^{\infty}\underset{s=-l}{\rm Res}\l(\frac{x^s}{\zeta(s)}\frac{\Gamma(s)}{\Gamma(1+\tau+s)}\r)\\
&+ O\l(  \frac{x^{2}}{\tau {T_*}^{\tau}} + \frac{x^2{T_*}^{-1-\tau+\varepsilon}}{\log{x}} \r)
\end{align*}
for any $x > 0$. Let $T = x^{3/\tau}$. Now, we have
\begin{align*}
\sum_{|\gamma| < T}\frac{x^{i\gamma}}{\zeta'(\rho)}\frac{\Gamma(\rho)}{\Gamma(1+\tau+\rho)}
&\leq \left(\sum_{0 < \gamma < T}\frac{1}{\gamma^{1+\tau}|\zeta'(\rho)|^2}\right)^{1/2}\left(\sum_{0 < \gamma < T}\frac{1}{\gamma^{1+\tau}}\right)^{1/2}\\
&\ll 1
\end{align*}
for any $T > 0$ and $\tau \gg 1$ by assumption $J_{-1}(T) \ll T$.
Hence, for $\tau \gg 1$, we have
\begin{align*}
M_{\tau}(x) \ll x^{1/2}.
\end{align*}

Next, we show this corollary in the case $\tau = o(1)$.
Let $x > 2$ be a half integer. Then, by Lemma \ref{GRF}, we have
\begin{align*}
M_\tau(x) 	&= \frac{1}{2\pi i}\int_{2 -iT_*}^{2 + iT_*}\frac{x^s}{\zeta(s)}\frac{\Gamma(s)}{\Gamma(1+\tau+s)}ds 
				+ O\l(\frac{x^{3}}{{T_*}^{1+\tau}} \r).
\end{align*}
From this equation and equation (\ref{subexfoM}), we have
\begin{align*}
M_\tau(x) = &
\sum_{|\gamma| < T_*}\frac{x^{i\gamma}}{\zeta'(\rho)}\frac{\Gamma(\rho)}{\Gamma(1+\tau+\rho)}x^{1/2}
+\sum_{l=0}^{\infty}\underset{s=-l}{\rm Res}\l(\frac{x^s}{\zeta(s)}\frac{\Gamma(s)}{\Gamma(1+\tau+s)}\r)\\
&+ O\l(\frac{x^3}{{T_*}^{\tau+1}} + \frac{x^2{T_*}^{-1-\tau+\varepsilon}}{\log{x}} \r).
\end{align*}
Let $T = x^3$. Then
\begin{align*}
M_\tau(x) = &
\sum_{|\gamma| < T_*}\frac{x^{i\gamma}}{\zeta'(\rho)}\frac{\Gamma(\rho)}{\Gamma(1+\tau+\rho)}x^{1/2} + O(1)
\end{align*}
holds. 
Thanks to the Cauchy-Schwarz inequality, for any sufficiently large $T > 0$, we have
\begin{align}	\label{cor2IE1}
&\sum_{|\gamma| < T}\frac{x^{i\gamma}}{\zeta'(\rho)}\frac{\Gamma(\rho)}{\Gamma(1+\tau+\rho)}
\leq \left(\sum_{0 < \gamma < T}\frac{1}{\gamma^{1+\tau}|\zeta'(\rho)|^2}\right)^{1/2}\left(\sum_{0 < \gamma < T}\frac{1}{\gamma^{1+\tau}}\right)^{1/2}.
\end{align}
Using the assumption $J_{-1}(T) \ll T$ and Lemma \ref{gzdbz}, we obtain that
\begin{align*}
\sum_{0 < \gamma < T}\frac{1}{\gamma^{1+\tau}|\zeta'(\rho)|^2}
&= \frac{J_{-1}(T)}{T^{1+\tau}} + (1+\tau)\int_{1}^{T}\frac{J_{-1}(u)}{u^{2+\tau}}du\\
&\ll \frac{1}{T^{\tau}} + (1+\tau)\int_{1}^{T}\frac{du}{u^{1+\tau}}
\ll \frac{1}{T^{\tau}} + \frac{1+\tau}{\tau}\left( 1 - \frac{1}{T^{\tau}} \right), 
\end{align*}
and that
\begin{align*}
&\sum_{0 < \gamma < T}\frac{1}{\gamma^{1+\tau}}
= \frac{N(T)}{T^{1+\tau}} + (1+\tau)\int_{1}^{T}\frac{N(u)}{u^{2+\tau}}du\\
&= \frac{\log{T}}{2\pi T^{\tau}} + \frac{1+\tau}{2\pi}\int_{1}^{T}\frac{\log{u}}{u^{1+\tau}}du
 + \frac{1+\log{2\pi}}{2\pi}\int_{1}^{T}\frac{du}{u^{1+\tau}} + O(1)\\
&= \frac{\log{T}}{2\pi T^{\tau}} - \frac{1+\tau}{2\pi \tau T^{\tau}}\log{T} + \frac{1+\tau}{2\pi \tau^2}\left( 1 - \frac{1}{T^{\tau}} \right)
 + \frac{1+\log{2\pi}}{2\pi \tau}\left( 1 - \frac{1}{T^{\tau}} \right) + O(1).
\end{align*}
Hence we have
\begin{align}	\label{cor2IE4}
\sum_{0 < \gamma < T}\frac{1}{\gamma^{1+\tau}|\zeta'(\rho)|^2}
\ll \l\{ 
\begin{array}{ll}
	1			& \text{if \; $1 \ll \tau$}, \\
	\tau^{-1}	 	& \text{if \; $(\log{T})^{-1} \ll \tau = o(1)$}, \\
	\log{T} 		& \text{if \; $0 \leq \tau = o\l((\log{T})^{-1}\r)$},
\end{array}
\r.
\end{align}
and
\begin{align}	\label{cor2IE5}
\sum_{0 < \gamma < T}\frac{1}{\gamma^{1+\tau}}
\ll \l\{
\begin{array}{ll}
	1			& \text{if \; $1 \ll \tau$}, \\
	\tau^{-2} 		& \text{if \; $(\log{T})^{-1} \ll \tau = o(1)$}, \\
	(\log{T})^2 	& \text{if \; $0 \leq \tau = o\l((\log{T})^{-1}\r)$}.
\end{array}
\r.
\end{align}
Therefore, we obtain
\begin{align*}
\sum_{|\gamma| < T}\frac{x^{i\gamma}}{\zeta'(\rho)}\frac{\Gamma(\rho)}{\Gamma(1+\tau+\rho)} 
\ll \l\{
\begin{array}{ll}
	1				& \text{if \; $1 \ll \tau$}, \\
	\tau^{-3/2}		   	& \text{if \; $(\log{x})^{-1} \ll \tau = o(1)$}, \\
	(\log{x})^{3/2} 		& \text{if \; $0 \leq \tau = o\l((\log{x})^{-1}\r)$}
\end{array}
\r.
\end{align*}
for any half integer $x > 0$ by (\ref{cor2IE1}), (\ref{cor2IE4}) and (\ref{cor2IE5}).
Now, let $x$ be any positive number and $x_0$ be a half integer with $x-1 < x_0 \leq x$.

Using the Taylor expansion 
$\d{
(1+x)^{\alpha} 
= \sum_{n=0}^{\infty}
				\begin{pmatrix}
					\alpha \\
					n
				\end{pmatrix}
x^{n}},
$
we have the inequality
\begin{align*}
\left( 1 - \frac{n}{x} \right)^{\tau} - \left( 1 - \frac{n}{x_0} \right)^{\tau} \ll \frac{n2^{\tau}}{x^2}.
\end{align*}
Thus we find that
\begin{align*}
\frac{1}{\Gamma(\tau+1)}\sum_{n \leq x}\mu(n)\left( \left( 1 - \frac{n}{x} \right)^{\tau} - \left( 1 - \frac{n}{x_0} \right)^{\tau} \right)
\ll 2^{\tau}.
\end{align*}
Hence we obtain Corollary \ref{dzb} when $2^{\tau} \ll 1$.
\end{proof}

\begin{proof}[Proof of Corollary \ref{AWMH}]

We assume the weak Mertens Hypothesis. 
It is known (see \cite[Section 14.28]{T}) that the weak Mertens Hypothesis implies the Riemann Hypothesis, (SZC) and 
\begin{align*}
\sum_{|\gamma| \leq T}\frac{1}{|\rho\zeta'(\rho)|^2} \ll 1.
\end{align*}
For any $\tau \geq 1/2 + \varepsilon$, we have
\begin{align*}
&\sum_{|\gamma| \leq T}\left|\frac{x^{i\gamma}}{\zeta'(\rho)}\frac{\Gamma(\rho)}{\Gamma(\rho + \tau + 1)}\right|
\ll \sum_{|\gamma| \leq T}\frac{1}{|\zeta'(\rho)\rho^{1+\tau}|}\\
&\leq \l(\sum_{|\gamma| \leq T}\frac{1}{|\rho\zeta'(\rho)|^2} \r)^{1/2}\l( \sum_{|\gamma| \leq T}\frac{1}{|\rho|^{2\tau}} \r)^{1/2}
\ll_{\varepsilon} 1.
\end{align*}
Hence the series
\begin{align*}
\sum_{\gamma}\frac{x^{i\gamma}}{\zeta'(\rho)}\frac{\Gamma(\rho)}{\Gamma(\rho + \tau + 1)}
\end{align*}
is absolutely and uniformly convergent with respect to $x \in (0, \infty)$.

Therefore, we obtain Corollary \ref{AWMH} by Theorem \ref{exfoM}.
\end{proof}

\begin{proof}[Proof of Corollary \ref{MM}]

Assume the weak Mertens Hypothesis and let $\kappa$ be a real number. Then we have
\begin{align*}
\int_{1}^{x}\frac{M(u)}{u^{\kappa}}du 
&= \left[ \frac{uM_1(u)}{u^{\kappa}} \right]_{1}^{x} + \kappa\int_{1}^{x}\frac{M_{1}(u)}{u^{\kappa}}du\\
&= \kappa\int_{1}^{x}\frac{M_{1}(u)}{u^{\kappa}}du + x^{1-\kappa}M_1(x)
\end{align*}
since $\int_{0}^{x}M(u)du = xM_{1}(x)$.
When $\kappa > 1$, we obtain
\begin{align*}
&\int_{1}^{x}\frac{M_{1}(u)}{u^{\kappa}}du
=\int_{1}^{x}\sum_{\gamma}\frac{u^{\rho-\kappa}}{\zeta'(\rho)}\frac{\Gamma(\rho)}{\Gamma(\rho+2)}du
+\int_{1}^{x}\sum_{l=0}^{\infty}\underset{s=-l}{\rm Res}\l(\frac{u^s}{\zeta(s)}\frac{\Gamma(s)}{\Gamma(s + 2)}\r)\frac{du}{u^{\kappa}}\\
&= \sum_{\gamma}\frac{x^{\rho-\kappa+1}-1}{\zeta'(\rho)(\rho - \kappa + 1)\rho(\rho+1)}
+\int_{1}^{\infty}\sum_{l=0}^{\infty}\underset{s=-l}{\rm Res}\l(\frac{u^s}{\zeta(s)}\frac{\Gamma(s)}{\Gamma(s + 2)}\r)\frac{du}{u^{\kappa}}
+O\l( x^{1-\kappa} \r)
\end{align*}
by Corollary \ref{AWMH}. Hence we have 
\begin{align*}
\int_{1}^{x}\frac{M(u)}{u^{\kappa}}du
= \sum_{\gamma}\frac{x^{\rho-\kappa+1}}{\zeta'(\rho)\rho(\rho-\kappa+1)} + A(\kappa) + O\l(x^{1-\kappa}\r),
\end{align*}
where 
\begin{align*}
A(\kappa) = \kappa \int_{1}^{\infty}\sum_{l=0}^{\infty}\underset{s=-l}{\rm Res}\l(\frac{u^s}{\zeta(s)}\frac{\Gamma(s)}{\Gamma(s + 2)}\r)
\frac{du}{u^{\kappa}} - \kappa\sum_{\gamma}\frac{1}{\zeta'(\rho)(\rho-\kappa+1)\rho(\rho+1)}.
\end{align*}
Moreover, we find that
\begin{align*}
A(\kappa) = 
&\frac{10\kappa - 12}{\kappa-1} + \kappa \sum_{l = 1}^{\infty}\frac{2(-1)^{l}(2l-2)!(2\pi)^{2l}}{\{(2l)!\}^2
\left(2l + \kappa - 1 \right) \zeta(2l+1)}\\
&-\kappa\sum_{\gamma}\frac{1}{\zeta'(\rho)\rho(\rho+1)(\rho-\kappa+1)}
\end{align*}
by calculating the equation (\ref{Resint}) in the case $k = 1$.

The case $\kappa \leq 1$ is similar.
\end{proof}




					\section{\bf Proofs of Theorems \ref{divIM} and \ref{IM}} 


We shall prove the following lemmas.
Theorem \ref{divIM} is immediate consequence of these Lemma \ref{divIMLem1} and Lemma \ref{divIMLem2}.

\begin{lemma} \label{divIMLem1}
Let $\Theta$ denote the supremum of real parts of the non-trivial zeros of the Riemann zeta-function and $\kappa \leq 3/2$.
Then we have
\begin{align*}
\int_{1}^{x}\frac{M(u)}{u^{\kappa}}du - \frac{2}{\zeta(1/2)} = \Omega_{\pm}\left(x^{\Theta - \kappa + 1 - \varepsilon}\right)
\end{align*}
for any fixed number $\varepsilon > 0$, where the term $2 / \zeta(1/2)$ is missing unless $\kappa = 3/2$.
\end{lemma}

\begin{proof}

We denote
\begin{align} \label{H(x)}
H_{\kappa}(x) := \int_{1}^{x}\frac{M(u)}{u^{\kappa}}du - \frac{2}{\zeta(1/2)}, 
\end{align}
where the term $2/\zeta(1/2)$ is missing unless $\kappa = 3/2$.
First we show the case $\kappa < 3/2$.
Now, we suppose that there exists a number $0 < \varepsilon_0 < 3/2 - \kappa$ such that
\begin{align}	\label{assumption1}
H_{\kappa}(x) < x^{\Theta - \kappa + 1 - \varepsilon_0} \qquad (x > K(\varepsilon_0)).
\end{align}
We have
\begin{align*}
&\int_{1}^{\infty}\frac{x^{\Theta - \kappa + 1 - \varepsilon_0} - H_{\kappa}(x)}{x^{s-\kappa + 2}}dx\\
&= \frac{1}{s - \Theta + \varepsilon_0} 
+ \left[ \frac{1}{s - \kappa - 1}\frac{H_{\kappa}(x)}{x^{s-\kappa+1}} \right]_{1}^{\infty} 
- \frac{1}{s-\kappa+1}\int_{1}^{\infty}\frac{M(x)}{x^{s+1}}dx\\
&= \frac{1}{s - \Theta + \varepsilon_0} - \frac{1}{s(s-\kappa+1)\zeta(s)}
\end{align*}
for any $\sigma > 1$. Notice that the first integral term is absolutely convergent for $\sigma > \Theta - \varepsilon_0$ since
the last equation is definite value when $s = \sigma > \Theta - \varepsilon_0$, and we assume (\ref{assumption1}).
Hence, for $\sigma > \Theta - \varepsilon_0$, we have
\begin{align*}
\frac{1}{s(s-\kappa+1)\zeta(s)}
= \int_{1}^{\infty}\frac{H_{\kappa}(x) - x^{\Theta - \kappa + 1 - \varepsilon_0}}{x^{s-\kappa + 2}}dx
+ \frac{1}{s - \Theta + \varepsilon_0}.
\end{align*}
Now, we find that $\d{\frac{1}{s(s-\kappa+1)\zeta(s)}}$ is regular for $\sigma > \Theta - \varepsilon_0$ since
the right hand side is regular for $\sigma > \Theta - \varepsilon_0$ by assumption (\ref{assumption1}).
However, this is a contradiction with the definition $\Theta$.
Hence we have $H_{\kappa}(x) = \Omega_{+}(x^{\Theta - \kappa + 1 - \varepsilon})$ for any $\varepsilon > 0$. 
Similarly, we have $H_{\kappa}(x) = \Omega_{-}(x^{\Theta - \kappa + 1 - \varepsilon})$.

Next, we consider the case $\kappa = 3/2$. 
We suppose that there exists a number $0 < \varepsilon_0 < 1/4$ such that
\begin{align*}
H_{3/2}(x) < x^{\Theta - 1/2 - \varepsilon_0} \qquad (x > K(\varepsilon_0)).
\end{align*}
Then we have
\begin{align*}
\int_{1}^{\infty}\frac{x^{\Theta - 1/2 +\e_{0}} - H_{3/2}(x)}{x^{s+1/2}}dx
=\frac{1}{s - \Theta +\e_{0}} - \l(\frac{1}{s(s-1/2)\zeta(s)} -  \frac{2}{(s-1/2)\zeta(1/2)} \r).
\end{align*}
Hence we have the $\Omega$-result
\begin{align*}
H_{3/2}(x) = \Omega_{\pm}(x^{\Theta-1/2-\e})
\end{align*}
by considered in the same manner as in the case $\kappa < 3/2$.
\end{proof}

\begin{lemma} \label{divIMLem2}

We assume the Riemann Hypothesis.
Let $\rho = 1/2 + i\gamma$ be a non-trivial zero of the Riemann zeta-function and $\tau \leq 3/2$. Then we have
\begin{align*}
\limsup_{x \rightarrow \infty}\frac{1}{x^{3/2-\kappa}}\int_{1}^{x}\frac{M(u)}{u^{\kappa}}du - \frac{2}{\zeta(1/2)}
&\geq \frac{1}{|\rho(\rho-\kappa+1)\zeta'(\rho)|},\\
\liminf_{x \rightarrow \infty}\frac{1}{x^{3/2-\kappa}}\int_{1}^{x}\frac{M(u)}{u^{\kappa}}du - \frac{2}{\zeta(1/2)}
&\leq - \frac{1}{|\rho(\rho-\kappa+1)\zeta'(\rho)|},
\end{align*}
where if $\rho$ is a multiple zero, we consider as $\d{\frac{1}{|\rho\gamma\zeta'(\rho)|}} = +\infty$, and 
the term $\frac{2}{\zeta(1/2)}$ is missing unless $\kappa = 3/2$.
\end{lemma}

\begin{proof}

Let $s = \sigma + it$ be a complex number with $\sigma > 1/2$. 
Suppose that $H_{\kappa}(x)\leq cx^{3/2-\kappa}$ for all $x \geq X_0$. 
We define 
\begin{align*}
G_{\kappa}(s) &:= \int_{1}^{\infty}\frac{H_{\kappa}(x) - cx^{3/2-\kappa}}{x^{s + 2 - \kappa}}dx.
\end{align*}
Using integrate by parts, we have
\begin{align}	\label{EQG_{kappa}}
G_{\kappa}(s) 
&= \frac{1}{s(s + 1 -\kappa)\zeta(s)} - \frac{c}{(s - 1/2)}
\end{align}
for $\sigma > 1/2$ because $H_{\kappa}(x) \ll x^{3/2- \kappa + \varepsilon}$ and $M(x) \ll x^{1/2 + \varepsilon}$ hold under the Riemann Hypothesis.
Now, we define that
\begin{align*}
I_{\kappa}(s) := \int_{1}^{\infty}\frac{H_{\kappa}(x) - cx^{3/2-\kappa}}{x^{s + 2 - \kappa}}(1 + \cos(\phi - \gamma\log{x}))dx,
\end{align*}
where $\phi$ is any real number.
Then we find that
\begin{align} \label{EQI_{kappa}}
I_{\kappa}(s) = G_{\kappa}(s) + \frac{1}{2}\l( e^{i\phi}G_{\kappa}(s + i\gamma) + e^{-i\phi}G_{\kappa}(s-i\gamma) \r).
\end{align}
Now, we have
\begin{align}	\label{IEQI}
\limsup_{\sigma \rightarrow 1/2+0}(\sigma - 1/2)I_{\kappa}(\sigma) \leq 0
\end{align}
for any real $\phi$ since $\d{\limsup_{\sigma \rightarrow 1/2+0}I_{\kappa}(\sigma) \not= +\infty}$ holds by the definition $I_{\kappa}(s)$. 
But there exists $\phi$ such that
\begin{align*}
\lim_{\sigma \rightarrow 1/2 + 0}(\sigma - 1/2)I_{\kappa}(\sigma) = \infty
\end{align*}
if $\rho = \frac{1}{2} + i\gamma$ is a multiple zero of $\zeta(s)$.
Hence if $\zeta(s)$ has a multiple zero, then we have
\begin{align*}
\limsup_{x \rightarrow \infty}H_{\kappa}(x) = +\infty.
\end{align*}
Next we assume (SZC).
Then, by \eqref{EQG_{kappa}} and \eqref{EQI_{kappa}}, we have
\begin{align*}
\lim_{s \rightarrow 1/2}(s-1/2)I_{\kappa}(s) = \frac{2}{\zeta(1/2)} - c + \Re\left( \frac{e^{i\phi}}{\rho(\rho-\kappa+1)\zeta'(\rho)} \right),
\end{align*}
and if we take $\phi = \arg(\rho(\rho-\kappa+1)\zeta'(\rho))$, we have
\begin{align*}
\frac{2}{\zeta(1/2)} - c +  \frac{1}{|\rho(\rho-\kappa+1)\zeta'(\rho)|} \leq 0
\end{align*}
by \eqref{IEQI}, where $\frac{2}{\zeta(1/2)}$ is missing unless $\kappa = \frac{3}{2}$.
Hence we obtain
\begin{align*}
\frac{2}{\zeta(1/2)} + \frac{1}{|\rho(\rho-\kappa+1)\zeta'(\rho)|} \leq c.
\end{align*}
Hence we have
\begin{align*}
\limsup_{x \rightarrow \infty}\frac{H_{\kappa}(x)}{x^{3/2-\kappa}} \geq \frac{2}{\zeta(1/2)} + \frac{1}{|\rho(\rho-\kappa+1)\zeta'(\rho)|},
\end{align*}
where the term $\frac{2}{\zeta(1/2)}$ is missing unless $\kappa = \frac{3}{2}$.

Similarly, we have
\begin{align*}
\liminf_{x \rightarrow \infty}\frac{H_{\kappa}(x)}{x^{3/2-\kappa}} \leq \frac{2}{\zeta(1/2)} - \frac{1}{|\rho(\rho-\kappa+1)\zeta'(\rho)|},
\end{align*}
where the term $\frac{2}{\zeta(1/2)}$ is missing unless $\kappa = \frac{3}{2}$.
\end{proof}

\begin{proof}[Proof of Theorem \ref{IM}]

By Lemma \ref{divIMLem2}, we may assume (SZC). We denote
\begin{align*}
H_{\kappa}(x) &:= \int_{1}^{x}\frac{M(u)}{u^{\kappa}}du\\
G_{\kappa}(s) &:= \int_{1}^{\infty}\frac{H_{\kappa}(x) - cx^{3/2-\kappa}}{x^{s + 2 - \kappa}}dx\\
{I_{\kappa}}^*(s) &:= \int_{1}^{\infty}\frac{H_{\kappa}(x) - cx^{3/2-\kappa}}{x^{s+2-\kappa}}\prod_{k=1}^{K}(1 + \cos(\phi_{k} - \gamma_{k}\log{x}))dx
\end{align*}
as in the previous section. Then we find that
\begin{align*}
{I_{\kappa}}^*(s) = G_{\kappa}(s) + \frac{1}{2}\sum_{k=1}^{K}
\left( e^{i\phi_{k}}G_{\kappa}(s+i\gamma_{k}) + e^{-i\phi_{k}}G_{\kappa}(s-i\gamma_{k}) \right) + J_{\kappa}(s)
\end{align*}
where $J_{\kappa}(s)$ is a linear combination of $G_{\kappa}$ at arguments of the form 
\begin{align*}
	s + i\sum_{k=1}^{K}\varepsilon_{k}\gamma_{k}
\end{align*}
with more than one of ${\varepsilon}'_{k}$ are non-zero. 
Assuming the Linear Independence Conjecture, we see that the function $J_{\kappa}(s)$ does not have a pole at $s = 1/2$. 
Hence we have
\begin{align*}
\lim_{s \rightarrow 1/2}(s - 1/2){I_{\kappa}}^*(s) 
= \frac{2}{\zeta(1/2)} - c + \sum_{k=1}^{K}\Re\left( \frac{e^{i\phi_{k}}}{\rho_{k}(\rho_{k} - \kappa + 1)\zeta'(\rho_{k})} \right),
\end{align*}
where the term $\frac{2}{\zeta(1/2)}$ is missing unless $\kappa = \frac{3}{2}$.
Considering in the same manner as in the proof of Lemma \ref{divIMLem2}, we have 
\begin{align*}
\frac{2}{\zeta(1/2)} - c + \sum_{k=1}^{K} \frac{1}{|\rho_{k}(\rho_{k} - \kappa + 1)\zeta'(\rho_{k})|} \leq 0
\end{align*}
if we take $\phi_{k} = \arg(\rho_{k}(\rho_{k} - \kappa + 1)\zeta'(\rho_{k}))$.
Hence we have
\begin{align*}
\frac{2}{\zeta(1/2)} + \frac{1}{2}\sum_{\gamma} \frac{1}{|\rho(\rho - \kappa + 1)\zeta'(\rho)|} \leq c
\end{align*}
since $K$ is any positive integer.
Similarly, we have inequality (\ref{IM2}).
\end{proof}

\begin{acknow*}
The author expresses his gratitude to Prof.\ Kohji Matsumoto, Prof.\ Isao Kiuchi, Prof.\ Masatoshi Suzuki, Prof.\  Hirotaka Akatsuka,
Prof.\ Sumaia Saad Eddin, Mr.\  Tomohiro Ikkai and Mr.\  Sohei Tateno for their helpful comments.
\end{acknow*}




\end{document}